\documentclass[12pt]{amsart}
\advance\hoffset by -0.5 truecm
\usepackage[colorlinks=true]{hyperref}
\hypersetup{urlcolor=blue, citecolor=red}

\usepackage{euscript}

\newtheorem{Theorem}{Theorem}[section]
\newtheorem{Corollary}[Theorem]{Corollary}
\newtheorem{Lemma}[Theorem]{Lemma}
\newtheorem{Proposition}[Theorem]{Proposition}
\theoremstyle{definition}
\newtheorem{Definition}[Theorem]{Definition}
\newtheorem{Remark}[Theorem]{Remark}

\def \p{\partial}

\def \<{\langle}
\def \>{\rangle}

\def \tM{\widetilde{M}}

\begin{document}

\title[Invariant distributions and the geodesic ray transform]{Invariant distributions and the geodesic ray transform}

\author[G. P. Paternain]{Gabriel P. Paternain}
\address{Department of Pure Mathematics and Mathematical Statistics, University of Cambridge, Cambridge CB3 0WB, UK}
\email{g.p.paternain@dpmms.cam.ac.uk}

\author[H. Zhou]{Hanming Zhou}
\address{Department of Pure Mathematics and Mathematical Statistics, University of Cambridge, Cambridge CB3 0WB, UK}
\email{hz318@dpmms.cam.ac.uk}

\begin{abstract} We establish an equivalence principle between the solenoidal injectivity of the geodesic ray transform acting on symmetric $m$-tensors and the existence of invariant distributions or smooth first integrals with prescribed projection
over the set of solenoidal $m$-tensors. We work with compact simple manifolds, but several of our results
apply to non-trapping manifolds with strictly convex boundary.

\end{abstract}
\maketitle


\section{Introduction}

The present paper studies the geodesic ray transform of a compact simply connected Riemannian manifold with no conjugate points and strictly convex boundary. Our main objective is to establish an equivalence principle between injectivity of the ray transform acting on solenoidal symmetric $m$-tensors and the existence of solutions to the transport equation (associated with the geodesic vector field) with prescribed projection over the set of solenoidal $m$-tensors.

The Radon transform in the plane is the most fundamental example of geodesic ray transform. It packs the integrals of a function $f$ in $\mathbb{R}^2$ over straight lines:
$$
Rf(s,\omega) = \int_{-\infty}^{\infty} f(s\omega + t\omega^{\perp}) \,dt, \quad s \in \mathbb{R}, \omega \in S^1.
$$
Here $\omega^{\perp}$ is the rotation of $\omega$ by $90$ degrees counterclockwise. The properties of this transform are well studied \cite{Hel} and constitute the theoretical underpinnings for many medical imaging methods such as CT and PET.
Generalizations of the Radon transform are often needed. In seismic and ultrasound imaging one finds ray transforms where the measurements are given by integrals over more general families of curves, often modeled as the geodesics of a Riemannian metric. Moreover, integrals of tensor fields over geodesics are ubiquitous in rigidity questions in differential geometry and dynamics. 

In this paper we will relate the injectivity properties of the geodesic ray transform with a well studied subject in classical mechanics: the existence of special first integrals of motion along geodesics. Some Riemannian metrics admit distinguished first integrals, e.g. the geodesic flow of an ellipsoid in $\mathbb{R}^3$ admits a non-trivial first integral which is quadratic in momenta. As recently shown in \cite{KM} a generic metric does not admit a non-trivial first integral that is polynomial in momenta, but here we will show a complementary statement going in the opposite direction: from the injectivity of the geodesic ray transform on tensors we will show that it is possible to construct smooth first integral with {\it any} prescribed polynomial part. 
In other words, given a polynomial $F$ of degree $m$ in momenta satisfying a natural restriction condition (related with the transport equation, see Section 7)),
we will show that we can find a smooth function $G$ whose dependence on momenta is of order $>m$ such
that $F+G$ is a first integral of the geodesic flow. Generically $G$ is non-vanishing and {\it not} polynomial in momenta.

Let us now explain our results in more detail. The geodesic ray transform acts on functions defined on the unit sphere bundle of a compact oriented $n$-dimensional Riemannian manifold $(M,g)$ with boundary $\p M$ ($n\geq 2$).
Let $SM$ denote the unit sphere bundle on $M$, i.e.
$$SM:=\{(x,\xi)\in TM : \|\xi\|_g=1\}.$$
We define the volume form on $SM$ by $d\Sigma^{2n-1}(x,\xi)=|dV^n(x)\wedge d\Omega_x(\xi)|$, where $dV^n$ is the volume form on $M$ and $d\Omega_x(\xi)$ is the volume form on the fibre $S_xM$.
The boundary of $SM$ is $\p SM:=\{(x,\xi)\in SM : x\in \p M\}$. On $\p SM$ the natural volume form is $d\Sigma^{2n-2}(x,\xi)=|dV^{n-1}(x)\wedge d\Omega_x(\xi)|$, where $dV^{n-1}$ is the volume form on $\p M$. We define two subsets of $\p SM$
$$\p_{\pm}SM:=\{(x,\xi)\in \p SM : \pm\<\xi,\nu(x)\>_g\leq 0\},$$
where $\nu(x)$ is the outward unit normal vector on $\p M$ at $x$. It is easy to see that
$$\p_+SM\cap\p_-SM=S(\p M).$$

Given $(x,\xi)\in SM$, we denote by $\gamma_{x,\xi}$ the unique geodesic with $\gamma_{x,\xi}(0)=x$ and $ \dot{\gamma}_{x,\xi}(0)=\xi$ and let $\tau(x,\xi)$ be the first time when the geodesic $\gamma_{x,\xi}$  exits $M$. 

We say that $(M,g)$ is {\it non-trapping} if $\tau(x,\xi)<\infty$ for all $(x,\xi)\in SM$.

\begin{Definition}
The \emph{geodesic ray transform} of a function $f \in C^{\infty}(SM)$ is the function 
\begin{equation*}
If(x,\xi)=\int\limits_{0}^{\tau(x,\xi)}f(\gamma_{x,\xi}(t),\dot{\gamma}_{x,\xi}(t)) \,dt,\quad
(x,\xi)\in \partial_{+} SM.
\end{equation*}
\end{Definition}

Note that if the manifold $(M,g)$ is non-trapping and has strictly convex boundary, then $I:C^{\infty}(SM)\rightarrow C^{\infty}(\partial_{+} SM)$, and Santal\'o's formula (see Section 2) implies that $I$ is also a bounded map $L^{2}(SM)\rightarrow L^{2}_{\mu}(\partial_{+}SM)$, where  $d\mu(x,\xi)=|\langle \nu(x), \xi \rangle|d\Sigma^{2n-2}(x,\xi)$ and $L^{2}_{\mu}(\partial_{+}SM)$ is the space of functions on $\partial_{+}SM$ with inner product
$$(u,v)_{L^{2}_{\mu}(\partial_{+}SM)}=\displaystyle\int_{\partial_{+}SM}u\overline{v}\,d\mu.$$


Given $f \in C^{\infty}(SM)$, what properties of $f$ may be determined from the knowledge of $If$?
Clearly a general function $f$ on $SM$ is not determined by its geodesic ray transform alone, since $f$ depends on more variables than $If$. In applications one often encounters the transform $I$ acting on special functions on $SM$ that arise from symmetric tensor fields, and we will now consider this case.

We denote by $C^{\infty}(S^m(T^*M))$ the space of smooth covariant symmetric tensor fields of rank $m$ on $M$
with $L^2$ inner product:
$$(u,v):=\int_M u_{i_1\cdots i_m}\overline{v^{i_1\cdots i_m}} \,dV^n,$$
where $v^{i_1\cdots i_m}=g^{i_1j_1}\cdots g^{i_mj_m} v_{j_1\cdots j_m}$. There is a natural 
map 
$$\ell_{m}: C^{\infty}(S^{m}(T^*M))\to C^{\infty}(SM)$$ given by $\ell_{m}(f)(x,\xi):=f_{x}(\xi,\dots,\xi)$.
We can now define the geodesic ray transform acting on symmetric $m$-tensors simply by
setting $I_{m}:=I\circ \ell_{m}$.
Let $d=\sigma\nabla$ be the symmetric inner differentiation, where $\nabla$ is the Levi-Civita connection associated with $g$, and $\sigma$ denotes symmetrization. 
It is easy to check that if $v=dp$ for some $p\in C^{\infty}(S^{m-1}(T^*M))$ with $p|_{\p M}=0$, then $I_{m}v=0$. The tensor tomography problem asks the following question: are such tensors the only obstructions for $I_{m}$ to be injective? If this is the case, then we say $I$ is {\it solenoidal injective or s-injective} for short. The problem is wide open for
compact non-trapping manifolds with strictly convex boundary (but see \cite{UV,SUV}). There are more results if one assumes the stronger condition of being {\it simple}, i.e., $(M,g)$ is simply connected, has no conjugate points and strictly convex boundary.  For simple surfaces, the tensor tomography problem has been completely solved \cite{PSU1}. For simple manifolds of any dimension, solenoidal injectivity is known for $I_{0}$ and $I_{1}$ \cite{Mu, AR}.
 For $m$-tensors, $m\geq 2$, the tensor tomography problem is still open, but some substantial partial results were established under additional assumptions, see e.g. \cite{PS, Sh1, SU1, PSU5, SUV}. 

Let us explain a bit further the term ``solenoidal injective''. Consider the  Sobolev space $H^k(S^m(T^*M))$
naturally associated with the $L^{2}$ inner product defined above.
 By \cite{Sh1, SSU}, there is an orthogonal decomposition of $L^2$ symmetric tensors fields. Given $v\in H^k(S^m(T^*M)),\, k\geq 0$, there exist uniquely determined $v^s\in H^k(S^m(T^*M))$ and $p\in H^{k+1}(S^{m-1}(T^*M))$, such that 
$$v=v^s+dp,\quad \delta v^s=0,\quad p|_{\p M}=0,$$
where $\delta$ is the divergence. We call $v^s$ and $dp$ the {\it solenoidal} part and {\it potential} part of $v$ respectively. Moreover, we denote by $H^k(S_{sol}^m(T^*M))$ ($C^{\infty}(S_{sol}^m(T^*M))$) the subspace of $H^k(S^m(T^*M))$ ($C^{\infty}(S^m(T^*M))$), whose elements are solenoidal symmetric tensor fields. Solenoidal injectivity of $I_{m}$ simply means that $I_{m}$ is injective when restricted to $C^{\infty}(S_{sol}^m(T^*M))$.

Let $I^*$ denote the adjoint of $I$ using the $L^2$ inner products defined above, that is,
$$(Iu,\varphi)=(u,I^*\varphi),$$
for $u\in L^2(SM)$, $\varphi\in L^2_{\mu}(\p_+SM)$. A simple application of Santal\'o's formula yields:
\[I^*\varphi=\varphi^{\sharp},\]
where $\varphi^{\sharp}(x,\xi):=\varphi(\gamma_{x,\xi}(-\tau(x,-\xi)),\dot{\gamma}_{x,\xi}(-\tau(x,-\xi)))$ (see Section 2 for details).
Observe that by definition $\varphi^{\sharp}$ is constant along orbits of the geodesic flow.
If we are now interested in $I^*_{m}$ we note that
\[I^*_{m}=\ell^*_{m}\circ I^*\]
and hence we just need to compute $\ell^*_{m}$. This is easy (see Section 2) and one finds
$$L_{m}f:=\ell_{m}^*f(x)_{i_1\cdots i_m}:=g_{i_1 j_1}\cdots g_{i_m j_m}\int_{S_xM}f(x,\xi)\xi^{j_1}\cdots\xi^{j_m}\,d\Omega_x(\xi).$$

The fundamental microlocal property of the geodesic ray transform is that, for simple manifolds, $I_{m}^*I_{m}$ is a pseudodifferential operator of order $-1$ on a slightly larger open manifold engulfing $M$. Moreover, it has a suitable ellipticity property when acting on solenoidal tensors \cite{SSU}. This has been exploited to great effect to derive surjectivity of $I_{m}^*$ knowing injectivity of $I_{m}$ \cite{PU, DU} for $m=0,1$. Since the range of $I_{m}^*$ is contained in the space of solenoidal tensors, by saying $I_{m}^*$ is {\it surjective} we mean that the range of $I^*_{m}$ equals the latter. Surjectivity of $I^*_{m}$ for tensors of order 0 and 1 has been the key for the recent success in the solution of several long standing questions in 2D \cite{SaU,PU,PSU1,PSU2,PSU4,G}. However, very little is known about surjectivity for $m\geq 2$ and this largely motivates the present paper.


The surjectivity properties of the adjoint of the geodesic ray transform reveal themselves in the existence of solutions $f$ to the transport equation $Xf=0$ with prescribed values for $L_{m}f$ in the space of solenoidal tensors.
Here $X$ is the geodesic vector field acting on distributions by duality (recall that $X$ preserves the volume form $d\Sigma^{2n-1}$). A distribution $f$ on $SM$ is said to be {\it invariant} if it satisfies $Xf=0$.
As we already mentioned, in this paper we mainly study the relation among the injectivity of $I_{m}$, the surjectivity of its adjoint $I_{m}^*$ on solenoidal tensor fields and the existence of some invariant distributions
or smooth first integrals associated with solenoidal tensor fields.  On a compact non-trapping manifold with strictly convex boundary, the geodesic ray transform $I_{m}$ is extendible to a bounded operator
\[I_{m}:H^{k}(S^{m}(T^*M))\to H^{k}(\partial_{+}SM)\]
for all $k\geq 0$ \cite[Theorem 4.2.1]{Sh1}. Moreover, it can be easily checked that 
$$I_{m}(H^{k}_{0}(S^{m}(T^*M)))\subset H^{k}_{0}(\partial_{+}SM)$$ and hence we can define $I_{m}^*$ by duality
acting on negative Sobolev spaces to obtain a bounded operator:
\[I^*_{m}:H^{-k}(\partial_{+}SM)\to H^{-k}(S^{m}(T^*M)).\]
In other words, for $\varphi\in H^{-k}(\partial_{+}SM)$, $I^*_{m}\varphi$ is defined
by $(I^*_{m}\varphi,u)=(\varphi,I_{m}u)$ for all $u\in H^{k}_{0}(S^{m}(T^*M))$.
Let $C_{\alpha}^{\infty}(\partial_{+}SM)$ denote the set of smooth functions $\varphi$ for which
$\varphi^{\sharp}$ is also smooth.

Our main result is the following theorem:

\begin{Theorem}\label{Main}
Let $M$ be a compact simple Riemannian manifold, then the following are equivalent:\\
(1) $I_{m}$ is s-injective on $C^{\infty}(S^m(T^*M))$;\\
(2) for every $u\in L^2(S^m_{sol}(T^*M))$, there exists $\varphi\in H^{-1}(\p_+SM)$ such that $u=I_{m}^*\varphi$.\\
(3) for every $u\in L^2(S^m_{sol}(T^*M))$, there exists $f\in H^{-1}(SM)$ satisfying $Xf=0$ and 
$u=L_{m}f.$\\
(4) for every $u\in C^{\infty}(S^{m}_{sol}(T^*M))$ there exists $\varphi\in C_{\alpha}^{\infty}(\partial_{+}SM)$
such that $u=I^*_{m}\varphi$.\\
(5) for every $u\in C^{\infty}(S^{m}_{sol}(T^*M))$ 
there exists $f\in C^{\infty}(SM)$ with $Xf=0$ such that $L_{m}f=u$.
\label{thm:main}
\end{Theorem}

We observe that by \cite[Theorem 1.1]{SSU}, s-injectivity of $I_{m}$ on $L^{2}(S^m(T^*M))$ is equivalent
to s-injectivity of $I_{m}$ on $C^{\infty}(S^m(T^*M))$.

Let us return to the subject of special first integrals associated with the geodesic flow. By considering the vertical Laplacian $\Delta$ on each fibre $S_{x}M$ of $SM$ we have a natural $L^2$-decomposition $L^{2}(SM)=\oplus_{m\geq 0}H_m(SM)$
into vertical spherical harmonics. We set $\Omega_m:=H_{m}(SM)\cap C^{\infty}(SM)$. Then a function $u$ belongs to $\Omega_m$ if and only if
$\Delta u=m(m+n-2)u$ where $n=\dim M$. 
The maps
$$\ell_m: C^{\infty}(S^m(T^*M))\to \bigoplus_{k=0}^{[m/2]}\Omega_{m-2k},$$ 
and
$$L_{m}:\bigoplus_{k=0}^{[m/2]}\Omega_{m-2k}\to C^{\infty}(S^m(T^*M))$$
are isomorphisms.
These maps give natural identification between functions in $\Omega_m$ and {\it trace-free} symmetric $m$-tensors (for details on this see \cite{GK2,DS,PSU5}).
If $(M,g)$ is a simple manifold with $I_{m}$ s-injective, item (5) in Theorem \ref{thm:main} is saying that given any $u\in C^{\infty}(S^{m}_{sol}(T^*M))$ there is a first integral of the geodesic flow $f$ such that $L_{m}f=u$. In other words if we let $F=L_{m}^{-1}u\in \bigoplus_{k=0}^{[m/2]}\Omega_{m-2k}$ and $G=f-F$, we see that $F$ is polynomial of degree $m$ in velocities and it can be completed by adding $G$ to obtain a first integral. We also see that (taking even or odd part of $f$ if necessary) $G\in \bigoplus_{k\geq 1}\Omega_{m+2k}$. These were the functions mentioned earlier in the introduction. If $G$ were to be zero, then there would be a first integral that is polynomial in velocities and generically these do not exist. Passing by we note that the paper \cite{PSU5} also constructs invariant distributions (they are not smooth in general) with prescribed $m$-th polynomial component using a different method (a Beurling transform), but it requires non-positive curvature for it to work. As already mentioned, here we use instead the normal operator $I^*_{m}I_{m}$.

The results in \cite{PU,DU} prove that (1) implies (4) or (5) in Theorem \ref{thm:main} for $m=0,1$, so the main contribution in the theorem is to cover the case $m\geq 2$ and also to provide additional invariant distributions associated with $L^2$ solenoidal tensors. The proof of Theorem \ref{thm:main} relies on a solenoidal extension of tensor fields.  For $m=0$ no extension is needed and for $m=1$ the situation is considerably simpler and an extension result is already available in \cite{KMPT}.
Paradoxically the need for a solenoidal extension does not arise in the more complicated setting of Anosov manifolds since there is no boundary. In this setting an analogous result to Theorem \ref{thm:main} (in the $L^2$-setting) has been recently proved by
C. Guillarmou in \cite[Corollary 3.7]{G} and these ideas gave rise to a full solution to the tensor tomography problem on an Anosov surface.

Since in $2D$ the tensor tomography problem has been fully solved \cite{PSU1} we derive:

\begin{Corollary} Let $(M,g)$ be a compact simple surface. For every $u\in C^{\infty}(S^{m}_{sol}(T^*M))$ 
there exists $f\in C^{\infty}(SM)$ with $Xf=0$ such that $L_{m}f=u$.
\label{thm:2D}
\end{Corollary}

We shall also give an alternative proof of the corollary using results from \cite{PSU3}. The alternative proof
avoids the smooth solenoidal extension and sheds some light into the relationship between the transport equation
and the solenoidal condition.

The rest of the paper is organized as follows. Section 2 contains some preliminaries. In Section 3 we establish the $L^2$  and $C^{\infty}$ compactly supported solenoidal extension of tensor fields. This necessitates at some point of the generic non-existence of non-trivial Killing tensor fields recently proved in \cite{KM}. Section 4 uses the well-established microlocal analysis to prove a surjectivity result for $I^*_{m}I_{m}$ following the strategy in \cite{DU}.
Section 5 establishes various boundedness properties on Sobolev spaces that allows us to extend
the relevant operators to negative Sobolev spaces (i.e. distributions). Section 6 bundles up everything
together and proves Theorem \ref{thm:main}. The final Section 7 gives an alternative proof of Corollary \ref{thm:2D} and clarifies the connection between solenoidal tensors and the transport equation.

\bigskip

\noindent {\bf Acknowledgements.} \  We are very grateful to Mikko Salo and Gunther Uhlmann for several discussions related to this paper; our main result  is motivated by a question raised in
\cite[Remark 11.7]{PSU5}.
We are also very grateful to Boris Kruglikov and Vladimir Matveev for extensive discussions related to Killing tensors and for the input provided in \cite{KM}. 
Finally we would like to thank Colin Guillarmou and Plamen Stefanov for invaluable comments on the first draft.
In particular, Stefanov provided us with the proof of Lemma \ref{local solenoidal extension} and Guillarmou
with the first part of Remark \ref{rem:colin}.
This research was supported by EPSRC grant EP/M023842/1.


\section{preliminaries}

In this section we provide details about the regularity properties of the operators introduced in the previous section. First we describe the basic notation we will use frequently in the rest of the paper. Given a compact Riemannian manifold $\mathcal{M}$ with boundary, we define:
$$C^{\infty}_c(\mathcal{M}^{int}):=\{f\in C^{\infty}(\mathcal{M}): \mbox{supp}\,f\subset \mathcal{M}^{int}\};$$
$$H^{k}_c(\mathcal{M}^{int}):=\{f\in H^{k}(\mathcal{M}): \mbox{supp}\,f\subset \mathcal{M}^{int}\},\quad \mbox{for}\, k\in \mathbb{Z}.$$
Then for any $s>0, s\in\mathbb{Z}$, $H^s_0(\mathcal{M})$ is the completion of $C^{\infty}_c(\mathcal{M}^{int})$ under the $H^s$ norm. 

Now let $M$ be a compact manifold, given $f\in C^{\infty}(SM),\, u\in C^{\infty}(S^m(T^*M))$, then
\begin{equation*}
\begin{split}
(\ell_m u, f) & =\int_{SM} u_{j_1\cdots j_m}(x)\xi^{j_1}\cdots \xi^{j_m} f(x,\xi)\, d\Sigma^{2n-1} \\
& =\int_{M} u_{j_1\cdots j_m}(x) \int_{S_xM} f(x,\xi)\xi^{j_1}\cdots \xi^{j_m}\, d\Omega_x(\xi) \, dV^n(x).
\end{split}
\end{equation*}
This means that 
$$L_m=\ell_m^*: C^{\infty}(SM)\to C^{\infty}(S^m(T^*M))$$
is given by $$L_m f (x)_{i_1\cdots i_m}=g_{i_1 j_1}\cdots g_{i_m j_m}\int_{S_xM} f(x,\xi)\xi^{j_1}\cdots \xi^{j_m}\, d\Omega_x(\xi).$$
Since the metric tensor $g$ is smooth, for the sake of simplicity, we identify $L_mf$ with its dual,
$$L_mf(x)^{j_1\cdots j_m}=\int_{S_xM}f(x,\xi)\xi^{j_1}\cdots\xi^{j_m}\,d\Omega_x(\xi).$$

On the other hand, it is easy to see that the map $\ell_m$ can be extend to the bounded operator
$$\ell_m: H^k(S^m(T^*M))\to H^k(SM)$$
for any integer $k\geq 0$. In particular $\ell_m(H_0^k(S^m(T^*M)))\subset H^k_0(SM)$. Therefore we can define 
\begin{equation}\label{L on -k}
L_m: H^{-k}(SM)\to H^{-k}(S^m(T^*M))
\end{equation}
in the sense of distributions and it is bounded.

Next, if $M$ is compact non-trapping with strictly convex boundary, we study the properties of $I$ and its adjoint $I^*$. Recall a useful integral identity called Santal\'o's formula. 
\begin{Lemma}{\cite[Lemma 3.3.2]{Sh2}}
Let $M$ be a compact non-trapping Riemannian manifold with strictly convex boundary. For every function $f\in C(SM)$ the equality
$$\int_{SM}f(x,\xi)\, d\Sigma^{2n-1}(x,\xi)=\int_{\partial_+SM}\,d\mu(x,\xi) \int_0^{\tau(x,\xi)}f(\gamma_{x,\xi}(t),\dot{\gamma}_{x,\xi}(t))\,dt$$
holds.
\end{Lemma}
Notice that the definition of compact dissipative Riemannian manifold (CDRM) in \cite{Sh2} is equivalent to compact non-trapping manifolds with strictly convex boundary.

Now let $\varphi\in C^{\infty}_{\alpha}(\p_+SM),\, f\in C^{\infty}(SM)$, by Santal\'o's formula
\begin{equation*}
\begin{split}
(If,\varphi)&=\int_{\p_+SM}\varphi(x,\xi)\,d\mu\int_0^{\tau(x,\xi)}f(\gamma_{x,\xi}(t), \dot{\gamma}_{x,\xi}(t))\, dt\\
&=\int_{\p_+SM}\,d\mu\int_0^{\tau(x,\xi)}\varphi^{\sharp}(\gamma_{x,\xi}(t),\dot{\gamma}_{x,\xi}(t))f(\gamma_{x,\xi}(t), \dot{\gamma}_{x,\xi}(t))\, dt\\
&=\int_{SM}\varphi^{\sharp}f\,d\Sigma^{2n-1}.
\end{split}
\end{equation*}
Thus $I^*\varphi=\varphi^{\sharp}$ with
$$I^*: C^{\infty}_{\alpha}(\p_+SM)\to C^{\infty}(SM)$$
bounded. By the proof of \cite[Theorem 4.2.1]{Sh1}, one can extend $I$ to a bounded operator 
$$I: H^k(SM)\to H^k(\p_+SM)$$ and $I(H^k_0(SM))\subset H^k_0(\p_+SM)$ for any integer $k\geq 0$ (notice that $I(C^{\infty}_c((SM)^{int}))\subset C^{\infty}_c((\p_+SM)^{int})$). Thus we can define the bounded operator
\begin{equation}\label{phi on -k}
I^*: H^{-k}(\p_+SM)\to H^{-k}(SM)
\end{equation}
in the sense of distributions.

Now given $u\in H_0^k(S^m(T^*M))$ and $\varphi\in H^{-k}(\p_+SM)$, $I_m^*\varphi$ is defined in the sense of distributions
$$(I_m^*\varphi,u):=(I^*\varphi, \ell_m u)=(\varphi, I\circ \ell_m u)=(\varphi, I_mu).$$

\begin{Lemma}
Given a compact non-trapping Riemannian manifold $M$ with strictly convex boundary,
$$I_m^*=L_m\circ I^*: H^{-k}(\p_+SM))\to H^{-k}(S^m(T^*M))$$
is a bounded operator.
\end{Lemma}

To conclude this section, we briefly discuss $X$, the generating vector field of the geodesic flow on the unit sphere bundle $SM$, acting on distributions. Since $X$ is a differential operator on $SM$, it is obvious that 
$$X: H^{k+1}(SM)\to H^{k}(SM) \quad k\geq 0.$$
For $f\in H^{-k}(SM)$ and $h\in H^{k+1}_0(SM)$ (so $Xh\in H^{k}_0(SM)$), we define $Xf\in H^{-k-1}(SM)$ in the sense of distributions (notice that the volume form $d\Sigma^{2n-1}$ is invariant under the geodesic flow)
$$(Xf,h):=(f,-Xh).$$


\section{Solenoidal extensions}
In the paper \cite{KMPT}, the authors proved the existence of compactly supported solenoidal extensions of solenoidal 1-forms to some larger manifold in both $L^2$ and smooth cases, namely

\begin{Proposition}
Let $\Omega$ be a bounded simply connected domain, with smooth boundary, contained in some Riemannian manifold $\mathcal{M}$. Let $U$ be an open neighborhood of $\Omega$ with $\p U$ smooth. Then there exists a bounded map $\mathcal{E}: L^2_{sol}(T^*\Omega)\to L^2_{U, sol}(T^*\mathcal{M})$, such that $\mathcal{E}|_{\Omega}=Id$. Moreover, $\mathcal{E}(C^{\infty}_{sol}(T^*\Omega))\subset C^{\infty}_{U,sol}(T^*\mathcal{M})$.
\end{Proposition}

Here $L^2_{U, sol}(T^*\mathcal{M})$ ($C^{\infty}_{U,sol}(T^*\mathcal{M})$) denotes the subspace of $L^2_{sol}(T^*\mathcal{M})$ ($C^{\infty}_{sol}(T^*\mathcal{M})$), consisting of elements suppoted in $U$.

Our goal is to extend this result to symmetric tensor fields of higher rank. However, for tensor fields of higher rank, new ideas are required and the argument is more involved.

\subsection{$L^{2}$ solenoidal extensions} We first prove the extension in the $L^2$ category by solving
a suitable elliptic system.

\begin{Proposition}\label{extension}
Let $\Omega$ be a bounded simply connected domain, with smooth boundary, contained in some Riemannian manifold $(\mathcal{M}, g)$. Let $U$ be an open neighborhood of $\Omega$ with $\p U$ smooth. Then given $m\geq 2$, $K\geq 2$ and $\epsilon>0$, there exist a Riemannian metric $\tilde{g}$ and a bounded map $\mathcal{E}: L^2(S^m_{sol}(T^*_{g}\Omega))\to L^2(S^m_{U, sol}(T^*_{\tilde{g}}\mathcal{M}))$ such that $\|\tilde{g}-g\|_{C^{K}}<\epsilon$, $\tilde{g}|_{\overline{\Omega}}=g$ and $\mathcal{E}|_{\Omega}=Id$.
\end{Proposition}
\begin{proof}
Given $u\in L^2(S^m_{sol}(T^*_g\Omega))$, i.e. $\delta u=0$ in the sense of distributions. By the Green's formula for symmetric tensor fields (cf. \cite{Sh1}) one can define the boundary contraction of $u$ with the outward unit normal vector $\nu$ on $\p\Omega$ in the sense of distributions, i.e. for $v\in H^1(S^{m-1}(T_g^*\Omega))$ we have the following equation
\begin{equation}\label{green}
( u, dv )_{\Omega}=( j_{\nu}u, v)_{\p\Omega}.
\end{equation}
Since the trace operator $T: H^1(S^{m-1}(T_g^*\Omega))\to H^{1/2}(S^{m-1}(\p T_g^*\Omega)),\, Tv=v|_{\p \Omega}$ is surjective, then $j_{\nu}u\in H^{-1/2}(S^{m-1}(\p T_g^*\Omega))$ is well-defined, and in local coordinates 
$$(j_{\nu}u)_{i_1i_2\cdots i_{m-1}}=u_{i_1i_2\cdots i_{m-1}j}\nu^j.$$
By \eqref{green}, for $v\in H^1(S^{m-1}(T_g^*\Omega))$ with $dv=0$ (Killing tensor fields on $\Omega$), we have $( j_{\nu}u ,v )_{\p\Omega}=0$.

It is known that generic (in the $C^K$-topology for $K\geq 2$) metrics admit only trivial integrals polynomial in momenta \cite{KM}, i.e. for a generic metric $h$ the only Killing tensor fields are of the form $c h^k$, where $c\in\mathbb{R}$ and $$h^k=\sigma(h\underbrace{\otimes\cdots\otimes }_{k}h)$$
is the symmetric tensor product of $k$ copies of $h$. Thus given any $\epsilon>0$ and $K\geq 2$, there is a smooth metric $\tilde{g}$ with $\|\tilde{g}-g\|_{C^{K}}<\epsilon$, $\tilde{g}|_{\overline{\Omega}}=g$, so that $(U\backslash\overline{\Omega}, \tilde{g})$ (thus $(U,\tilde g)$) does not have non-trivial Killing tensor fields.

Define 
\begin{equation*}
f=\left\{ \begin{array}{ll}
  -j_{\nu}u \, &\mbox{on} \, \p\Omega,\\
  0 \, &\mbox{on} \, \p U.
  \end{array} \right.
\end{equation*}
Denote $D:=U\backslash\overline{\Omega}$ and consider the following boundary value problem for systems of second order partial differential equations
\begin{equation}\label{system}
\left\{ \begin{array}{lr}
 \delta dw =0 \quad \mbox{in} \quad D, \\
  j_{\mu}dw =f\in H^{-1/2}(S^{m-1}(\p T^*_{\tilde{g}} D)),\\
   w\in H^1(S^{m-1}(T^*_{\tilde{g}}D)).
       \end{array} \right.
\end{equation}
Here $\mu$ is the outward unit normal vector on $\p D$ for $D$, notice $\mu|_{\p\Omega}=-\nu$. 
We claim that the system \eqref{system} is a regular elliptic system (also called coercive in some texts). Assume that the claim is true for the moment and let us continue the proof. 

Next, we study the solutions of the homogeneous problem. Let $\delta dv=0,\, j_{\mu}dv|_{\p D}=0$ for some $v\in H^1(S^{m-1}(T_{\tilde{g}}^*D))$, by ellipticity $v$ is smooth. Applying Green's formula, one has
$$\int_D\<dv,dv\>\,dV^n(x)=-\int_D\<\delta dv,v\>\,dV^n(x)+\int_{\p D}\<j_{\mu}dv,v\>\,dV^n(x)=0,$$
i.e. $dv\equiv 0$. So the solution set of the homogeneous problem is 
$$\mathcal{K}=\{v\in C^{\infty}(S^{m-1}(T^*_{\tilde{g}}D)): dv\equiv 0\},$$
the set of Killing tensor fields of rank $m-1$ on $D$.

Now by \cite[Theorem 4.11]{Mc}, \eqref{system} is solvable in $H^1(S^{m-1}(T^*_{\tilde{g}}D))$ for the given boundary condition $f$ if and only if $(v,f)_{\p D}=0$ for all $v\in \mathcal{K}$. Note that $(D,\tilde{g})$ does not have non-trivial Killing tensor fields. If $m$ is even, the only Killing $(m-1)$-tensor field is $v=0$, then $(v,f)_{\p D}=(0,f)_{\p D}=0$. If $m$ is odd, the Killing $(m-1)$-tensor fields in $D$ are of the form $v=c\tilde{g}^{(m-1)/2}|_D$. Thus we can extend $v$ to $v=c\tilde{g}^{(m-1)/2}|_U$, which is also a Killing tensor field in $\Omega$. By the definition of $f$, 
$$(v,f)_{\p D}=-(v,j_{\nu}u)_{\p\Omega}=-(v,\delta u)_{\Omega}-(dv,u)_{\Omega}=0,$$
since $\delta u=0,\,dv=0$ in $\Omega$.

Thus the system \eqref{system} is solvable. Let $w\in H^1(S^{m-1}(T^*_{\tilde{g}}D))$ be a solution of of \eqref{system} (the set of all solutions is $w+\mathcal{K}$) and define 
\begin{equation*}
\mathcal{E}u=\left\{ \begin{array}{ll}
  u \, &\mbox{in} \, \Omega,\\
  dw \, &\mbox{in} \, D,\\
  0 \, &\mbox{in} \, \mathcal{M}\backslash\overline{U}.
  \end{array} \right.
\end{equation*}
It is easy to see that $\mathcal{E}u\in L^2(S^m(T^*_{\tilde{g}}\mathcal{M}))$ and supp $\mathcal{E}u\subset \overline{U}$. In particular, for $v\in H^1(S^{m-1}(T^*_{\tilde{g}}\mathcal{M}))$
\begin{equation*}
\begin{split}
(\delta\mathcal{E}u, v)_{\mathcal{M}} &=-(\mathcal{E}u, dv)_{\mathcal{M}}=-(dw, dv)_D-(u,dv)_{\Omega}\\
&=-(j_{\mu}dw, v)_{\p D}-(j_{\nu}u,v)_{\p\Omega}\\
&=-(-j_{\nu}u,v)_{\p\Omega}-(j_{\nu}u,v)_{\p\Omega}\\
& =0.
\end{split}
\end{equation*}
Thus $\mathcal{E}u$ is solenoidal in the sense of distributions, and $\mathcal{E}u\in L^2(S^m_{U,sol}(T^*_{\tilde{g}}\mathcal{M}))$.

Moreover, by \cite[Theorem 4.11]{Mc}, we have the following stability estimate
$$\|\mathcal Eu\|^2_{L^2(\mathcal M)}=\|u\|^2_{L^2(\Omega)}+\|dw\|^2_{L^2(D)}\leq \|u\|^2_{L^2(\Omega)}+C\|j_{\nu}u\|^2_{H^{-1/2}(\p\Omega)}\leq C'\|u\|^2_{L^2(\Omega)},$$
i.e. $\mathcal E$ is bounded.
\end{proof}

The only thing left to prove is the claim about ellipticity.

\begin{Lemma}
The system \eqref{system} above is a regular elliptic system.
\end{Lemma}
\begin{proof}
It is well-known that $\delta d$ is a self-adjoint elliptic operator, see for example \cite{Sh1}, we just need to show that the Neumann boundary value problem satisfies the Lopatinskii condition.

To check the Lopatinskii condition, we follow a similar procedure as in the proof of \cite[Theorem 3.3.2]{Sh1}. We choose a local coordinates $(x^1,x^2,\cdots,x^{n-1},$ $x^n=t\geq 0)$ in a neighbourhood $W$ of $x_0=(x',0)\in\p D$ in $D$ so that $\p D\cap W=\{t=0\}$ and $g_{ij}(x_0)=\delta_{ij}$. Denote $d_0=\sigma_p d$ and $\delta_0=\sigma_p\delta$, the principal symbols of $d$ and $\delta$ respectively, then we need to show that the following boundary value problem for systems of ordinary differential equations 
\begin{equation*}
\left\{ \begin{array}{lr}
  \delta_0(x',0,\xi',D_t)d_0(x',0,\xi',D_t)w(t)=0,\\
  j_{-\frac{\p}{\p t}}d_0(x',0,\xi',D_t)w(t)|_{t=0}=f_0
    \end{array} \right.
\end{equation*}
has a unique solution in $\mathcal{N}_+$ for all $\xi'\in \mathbb{R}^{n-1}\backslash \{0\}$ and $f_0\in S^{m-1}(\mathbb{R}^{n})$, symmetric $(m-1)$-tensors on $\mathbb{R}^n$. Here $D_t=-id/dt$, and for the sake of simplicity, we drop the space variables $(x',0)$ from the symbols so
\begin{equation*}
\begin{split}
\mathcal{N}_+:=\{w\in S^{m-1}(\mathbb R^n)|_{\{x'\}\times [0,\infty)}:\, \delta_0(\xi',D_t) & d_0(\xi',D_t)w=0 \,\mbox{and}\, w\, \mbox{decays rapidly}\\ & \mbox{together with all derivatives as}\, t\to +\infty\}.
\end{split}
\end{equation*}
Since the equation det$(\delta_0(\xi',\zeta)d_0(\xi',\zeta))=0$ has real coefficients with no real root for $\xi'\neq 0$, it is not difficult to see that dim$\,\mathcal{N}_+=$dim$\,S^{m-1}(\mathbb{R}^{n})$. Thus it is sufficient to show that the homogeneous problem 
\begin{equation}\label{homogeneous}
\left\{ \begin{array}{lr}
  \delta_0(\xi',D_t)d_0(\xi',D_t)w(t)=0,\\
  j_{-\frac{\p}{\p t}}d_0(\xi',D_t)w(t)|_{t=0}=0
    \end{array} \right.
\end{equation}
has only the zero solution in $\mathcal{N}_+$.

By a similar computation as in the proof of \cite[Theorem 3.3.2]{Sh1}, we have the following Green's formula. Let $v(t)\in C^{\infty}([0,\infty)\to S^{m}(\mathbb{R}^n))$ and $w(t)\in C^{\infty}([0,\infty)\to S^{m-1}(\mathbb{R}^n))$ such that both of them decay rapidly together with all derivatives as $t\to +\infty$. If $j_{-\frac{\p}{\p t}}v(0)=0$ (notice that different from \cite{Sh1}, here we use Neumann boundary condition at $t=0$) then 
\begin{equation}\label{Green}
\int_0^{\infty}\<\delta_0(\xi',D_t)v,w\>\,dt=-\int_0^{\infty}\<v,d_0(\xi',D_t)w\>\,dt.
\end{equation}
Now if $w(t)\in\mathcal{N}_+$ is a solution to \eqref{homogeneous}, let $v(t)=d_0(\xi',D_t)w(t)$, by \eqref{Green} we obtain 
$$d_0(\xi',D_t)w(t)=0.$$
Notice that
$$(d_0(\xi)w)_{i_1\cdots i_m}=\frac{i}{m}\sum_{k=1}^m\xi_{i_k}w_{i_1\cdots \widehat{i_k}\cdots i_m},$$
where the $\wedge$ over $i_k$ means this index is omitted. Let $i_m=n$ and $\xi=(\xi',D_t)$, we obtain the following system of first order ordinary differential equations
$$(d_0(\xi',D_t)w)_{ni_1\cdots i_{m-1}}=\frac{i}{m}\{(\ell+1)D_tw_{i_1\cdots i_{m-1}}+\sum_{i_k\neq n}\xi_{i_k}w_{n i_1\cdots \widehat{i_k}\cdots i_{m-1}}\}=0,$$
where $\ell=\ell(i_1,\cdots, i_{m-1})$ is the number of occurrences of the index $n$ in $(i_1,\cdots, i_{m-1})$. Since $\lim_{t\to +\infty}w(t)=0$, by the induction on $\ell$, the only solution to the above first order homogeneous system is $w\equiv 0$, and this shows that \eqref{system} satisfies the Lopatinskii condition.
\end{proof}

\subsection{Smooth solenoidal extensions} In this subsection we achieve $C^{\infty}$  solenoidal extensions for tensors of arbitrary rank. Observe that the approach we use is quite different from the one of \cite{KMPT}. 

\begin{Proposition}\label{smooth extension}
Let $\Omega$ be a bounded connected domain, with smooth boundary, contained in some Riemannian manifold $(\mathcal{M}, g)$. Let $U$ be an open neighborhood of $\Omega$ with $\p U$ smooth. Then given $m\geq 2$, $K\geq 2$ and $\epsilon>0$, there exist a Riemannian metric $\tilde{g}$ and a bounded map $\mathcal{E}: H^k(S^m_{sol}(T^*_{g}\Omega))\to L^2(S^m_{U, sol}(T^*_{\tilde{g}}\mathcal{M}))$ for some integer $k\geq 2$ such that $\|\tilde{g}-g\|_{C^{K}}<\epsilon$, $\tilde{g}|_{\overline{\Omega}}=g$, $\mathcal{E}|_{\Omega}=Id$ and $\mathcal{E}(C^{\infty}(S^m_{sol}(T^*_g\Omega)))\subset C^{\infty}(S^m_{U,sol}(T^*_{\tilde g}\mathcal{M}))$.
\end{Proposition}

To prove the proposition, we start with the following lemma on the existence of solenoidal extensions that might not be compactly supported.

\begin{Lemma}\label{local solenoidal extension}
Let $\Omega$ be a bounded connected domain, with smooth boundary, contained in some Riemannian manifold $(\mathcal{M}, g)$. There exists an open neighborhood $U$ of $\Omega$ such that every $u\in C^{\infty}(S^m_{sol}(T^*\overline{\Omega}))$ can be extend to $\tilde u\in C^{\infty}(S^m_{sol}(T^*U))$ with $\tilde u|_{\overline{\Omega}}=u$.
\end{Lemma}
\begin{proof}
Given $u\in C^{\infty}(S^m_{sol}(T^*\overline{\Omega}))$ i.e. $\delta u=0$, in local coordinates $u=u_{j_1\cdots j_m}dx^{j_1}\otimes\cdots\otimes dx^{j_m}$ and 
\begin{equation}\label{divergence}
(\delta u)_{i_1\cdots i_{m-1}}=g^{jk}\nabla_j u_{ki_1\cdots i_{m-1}}=0,
\end{equation}
where 
\begin{equation}\label{nabla}
\nabla_j u_{ki_1\cdots i_{m-1}}=\p_j u_{ki_1\cdots i_{m-1}}-\Gamma^{\ell}_{jk}u_{\ell i_1\cdots i_{m-1}}-\sum_{s=1}^{m-1}\Gamma^{\ell}_{ji_s}u_{\ell k i_1\cdots \hat{i_s}\cdots i_{m-1}}.
\end{equation}
Pick $x_0\in \p \Omega$, we follow the idea of the proof in \cite[Lemma 4.1]{SU1} and choose semigeodesic coordinates $(x^1,\cdots x^{n-1},x^n)=(x',x^n)$ near $x_0$ with $\p \Omega=\{x^n=0\}$ and $\p_n=\nu$ the unit outward (with respect to $\Omega$) vector normal to $\p \Omega$; thus 
$$g^{kn}=\delta^k_{n},\quad \Gamma^{n}_{kn}=\Gamma^{k}_{nn}=0,\quad \forall\, k=1,2,\cdots,n.$$

We extend the components $u_{j_1\cdots j_m}$, $j_s< n,\,\forall\, 1\leq s\leq m$, smoothly to $U$ (note that $U\backslash\overline{\Omega}$ is determined by the semigeodesic neighborhood of $\p\Omega$), and denote the extensions by $v_{j_1\cdots j_m}$. We will construct the other components in $\{x^n>0\}$ by induction on the number of appearances of $n$ in $j_1\cdots j_m$. By equations \eqref{divergence} and \eqref{nabla} if $i_1,\cdots,i_{m-1}< n$
\begin{equation}\label{occur once}
\begin{split}
\p_n & v_{n i_1\cdots i_{m-1}}-\sum_{s=1}^{m-1}\sum_{\ell<n}\Gamma_{ni_s}^{\ell}  v_{\ell n i_1\cdots \hat{i_s}\cdots i_{m-1}} -\sum_{j,k<n}g^{jk}\bigg(\Gamma^n_{jk}v_{n i_1\cdots i_{m-1}}+\sum^{m-1}_{s=1}\Gamma^n_{ji_s}v_{nki_1\cdots \hat{i_s}\cdots i_{m-1}}\bigg)\\
&=-\sum_{j,k<n}g^{jk}\bigg(\p_j v_{ki_1\cdots i_{m-1}}-\sum_{\ell<n}\Gamma^{\ell}_{jk}v_{\ell i_i\cdots i_{m-1}}-\sum_{s=1}^{m-1}\sum_{\ell<n}\Gamma^{\ell}_{ji_s}v_{\ell k i_1\cdots \hat{i_s}\cdots i_{m-1}}\bigg).
\end{split}
\end{equation}
Notice that the right side of \eqref{occur once} is known, so it gives a system of first order linear ODEs. Given the initial values $u_{ni_1\cdots i_{m-1}}(x',0)=v_{ni_1\cdots i_{m-1}}(x',0)$, there exists a unique solution to \eqref{occur once}. Thus we obtain continuous $v_{ni_1\cdots i_{m-1}}$ with $i_1,\cdots,i_{m-1}<n$ near $\p M$. In particular, $v_{ni_1\cdots i_{m-1}}(x',x^n)$ depends smoothly on $x'$, the first $n-1$ variables.

By differentiating \eqref{occur once} repeatedly with respect to $x^n$, we get that $\p_n^s v_{ni_1\cdots i_{m-1}}(x',x^n)$, $s\geq 0$ are continuous in $x^n\geq 0$ and smooth with respect to $x'$. Moreover, by \eqref{occur once} and the fact that $u$ is solenoidal we carry out an induction on $s$, so
\begin{equation*}
\begin{split}
& \p_n^s v_{ni_1\cdots i_{m-1}}(x',0)\\
= & G^s_{i_1\cdots i_{m-1}}(\p_n^{\ell} v_{nj_1\cdots j_{m-1}},\p_n^{\ell} v_{j_1\cdots j_m},\p_n^{\ell}\p_{k}v_{j_1\cdots j_m}; \ell<s; j_1,\cdots,j_{m-1},j_m,k<n)(x',0)\\
= & G^s_{i_1\cdots i_{m-1}}(\p_n^{\ell} u_{nj_1\cdots j_{m-1}},\p_n^{\ell} u_{j_1\cdots j_m},\p_n^{\ell}\p_{k}u_{j_1\cdots j_m}; \ell<s; j_1,\cdots,j_{m-1},j_m,k<n)(x',0)\\
= & \p_n^s u_{ni_1\cdots i_{m-1}}(x',0),
\end{split}
\end{equation*}
for all $s\geq 0$, i.e. $\p_n^s v_{ni_1\cdots i_{m-1}}$ are consistent with $\p_n^s u_{ni_1\cdots i_{m-1}}$ at $(x',0)$.

Next by induction on the number of appearances of $n$ and repeatedly using equations \eqref{divergence} and \eqref{nabla}, one can get unique $$v_{ni_1\cdots i_{m-1}},\, v_{nni_1\cdots i_{m-2}},\, \cdots,\, v_{n\cdots ni_1},\, v_{n\cdots n},$$ which together with their normal derivatives with respect to $x^n$ of all orders are continuous (smooth with respect to $x'$) and consistent with the corresponding $\p_n^m u_{j_1\cdots j_m}$ at $(x',0)$. Therefore we get a smooth solenoidal $m$-tensor
\begin{equation*}
\tilde u=\left\{ \begin{array}{ll}
  u \, &\mbox{on} \quad \overline{\Omega},\\
  v \, &\mbox{on} \quad U\backslash \overline{\Omega}.
  \end{array} \right.
\end{equation*}  
\end{proof}

\begin{proof}[Proof of Proposition \ref{smooth extension}] There exist two precompact open neighborhoods $V$, $U$ of $\Omega$ which satisfy 
$$\Omega\subset \overline{\Omega}\subset V\subset \overline V\subset U\subset \overline U\subset \mathcal M.$$
Given $u\in C^{\infty}(S^m_{sol}(T^*\Omega))$, by Lemma \ref{local solenoidal extension}, we can extend $u$ to get $u_V\in C^{\infty}(S^m_{sol}(T^*V))$ with $u_V|_{\overline{\Omega}}=u$. Then we extend $u_V$ to a smooth $m$-tensor $w$ on $\mathcal M$ with supp$\,w\subset U$. Let $f=\delta w$ and $D=U\backslash \overline{\Omega}$ open, so supp$\,f\subset U\backslash V\subset D$. 

Similar to the perturbation of metrics argument in the proof of Proposition \ref{extension}, given any $\epsilon>0$ and $K\geq 2$, there is a smooth metric $\tilde{g}$ with $\|\tilde{g}-g\|_{C^{K}}<\epsilon$, $\tilde{g}|_{V}=g$, so that $(D, \tilde{g})$ does not have non-trivial Killing tensor fields. Now if $m$ is even, the only Killing $(m-1)$-tensor field on $(D,\tilde g)$ is $v=0$, then 
$$(v,f)_{D}=(0,f)_{D}=0.$$
If $m$ is odd, Killing $(m-1)$-tensor fields on $(D, \tilde g)$ are of the form $v=c\tilde{g}^{(m-1)/2}|_D$. Thus we can extend $v$ to $v=c\tilde{g}^{(m-1)/2}|_U$, which is also a Killing tensor field in $\Omega$. By the Green's formula, 
$$(v,f)_{D}=(v,\delta w)_D=-(dv,w)_D+(v,j_{\mu}w)_{\p D}=-(v,j_{\nu}u)_{\p\Omega}=-(v,\delta u)_{\Omega}-(dv,u)_{\Omega}=0,$$
since $\delta u=0,\,dv=0$ in $\Omega$. Here $\mu=-\nu$ is the unit outward normal vector on $\p D$ and
$$(j_{\mu}w)_{i_1i_2\cdots i_{m-1}}=w_{i_1i_2\cdots i_{m-1}j}\mu^j.$$

Now by \cite[Theorem 1.3]{De}, there exist $u_D\in C^{\infty}(S^m(T^*\mathcal M))$ with supp$\,u_D\subset U\backslash \Omega$, such that $\delta u_D=-f$. It is not difficult to check that the symmetric differentiation $d$ satisfies the Kernel Restriction Condition (KRC) and the Asymptotic Poincar\'e Inequality (API) of \cite{De}.  We define $\mathcal E u=w+u_D$, then $\delta\mathcal E u=\delta w+\delta u_D=f-f=0$, i.e. $\mathcal E u\in C^{\infty}(S^m_{U,sol}(T^*_{\tilde g}\mathcal M))$. Moreover, $\mathcal E u|_{\Omega}=u$. 

The argument above gives a construction for compactly supported smooth solenoidal extensions. One can further check that the extension can be constructed in a stable way. In view of the ODEs \eqref{occur once}, the solution is controlled by the initial value and the nonhomogeneous term on the right side under Sobolev norms, see e.g. \cite{Han}. By induction on the number of appearances of $n$ and repeatedly differentiating \eqref{occur once}, we have that
$$\|u_V\|_{H^{1}(V\backslash \overline{\Omega})}\leq C(\|j_{\mu}u\|_{H^{k_1}(\p \Omega)}+\sum_{i_s<n}\|(u_V)_{i_1\cdots i_m}\|_{H^{k_2}(V\backslash \overline{\Omega})})$$
for some $k_1, k_2\geq 1$. Note that in boundary normal coordinates $\mu=-\p_n$, and we have full freedom to control the elements $(u_V)_{i_1\cdots i_m},\, i_s<n,\,\forall 1\leq s\leq m$ by $u|_{\Omega}$ due to the fact that $\delta$ is an underdetermined elliptic operator. Thus 
$$\|u_V\|_{H^{1}(V\backslash \overline{\Omega})}\leq C \|u\|_{H^{k}(\Omega)}$$
for some integer $k\geq 2$. Then $\|w\|_{H^{1}(U)}\leq C\|u\|_{H^{k}(\Omega)}$ by extending $u_V$ to $w$ in a stable way.

Next we control the $L^2$ norm of $u_D$. Roughly speaking, $u_D$ is the symmetric differentiation of some smooth $m-1$ tensor $p$, multiplied by a smooth nonnegative weight which vanishes exponentially at the boundary of $D$, concretely $u_D=\psi^2\phi^2dp$ with $\phi$ a boundary defining function on $D$ and $\psi$ vanishes exponentially at the boundary $\p D$. By \cite[Lemma 10.2]{De}, 
$\|p\|_{H^2_{\phi,\psi}(D)}\leq C \|\psi^{-2}\delta w\|_{L^2_{\psi}(D)}$, where $H^2_{\phi,\psi}$ and $L^2_{\psi}$ are some weighted Sobolev spaces, see \cite{De} for more details. Then one can check that the following inequality with unweighted Sobolev norms holds
$$\|u_D\|_{L^2(D)}\leq C \|w\|_{H^1(U)}.$$
Now we combine the estimates above to obtain 
$$\|\mathcal E u\|_{L^2(\mathcal M)}\leq C_1(\|w\|_{L^2(U)}+\|u_D\|_{L^2(D)})\leq C_2\|w\|_{H^1(U)}\leq C\|u\|_{H^k(\Omega)}.$$
 for some $C>0$ independent of $u$. Since $C^{\infty}(S^m_{sol}(T^*\Omega))$ is dense in $H^k(S^m_{sol}(T^*\Omega))$ under the $H^k$ norm, we can extend $\mathcal E$ to a bounded map from $H^k$ to $L^2$ with the same properties, which completes the proof.
\end{proof}

\begin{Remark}
{\rm We expect that the $L^2$ norm of $\mathcal E u$ can be bounded by the $L^2$ norm of $u|_{\Omega}$ through sharper estimates, similar to the result under the $L^2$ setting in Section 3.1. However, the $H^k$ space is enough for carrying out the argument under the smooth setting in the next section, see Lemma \ref{smooth normal}.}
\end{Remark}



\section{Surjectivity of the normal operator $I_m^*I_m$}

Since $M$ is simple we can consider an extension $\widetilde{M}$ of $M$ which is open ($\tM=\tM^{int}$) and whose compact closure is also simple. It is well-known that the normal operator $N=I_m^*I_m$ is a pseudodifferential operator of order $-1$ on $\tM$, see for example \cite{Sh1, SU04,SSU,SU08}. Below is a lemma that roughly speaking gives a right parametrix for $N$ on the space of solenoidal tensor fields. The proof is similar to \cite[Theorem 3.1]{SSU}.
\begin{Lemma}\label{partial parametrix}
Let $S$ be a parametrix for the operator $\delta d$. There exists a pseudodifferential operator $Q$ of order $1$ on the bundle of symmetric $m$-tensor fields, $S^m(T^*\tM)$, such that 
\begin{equation}\label{parametrix}
E=NQ+dS\delta+K,
\end{equation} 
where $E$ is the identity operator and $K$ is a smoothing operator.
\end{Lemma}
\begin{proof}
Let $\lambda(\xi)$ be the principal symbol of the pseudodifferential operator $N$, and 
$$S_{\xi}^m(T_x^*\tM)=\{u\in S^m(T^*_x\tM): j_{\xi}u=0\},$$
where $j_{\xi}=-i\sigma_p(\delta): S^m(T^*_x\tM)\to S^{m-1}(T^*_x\tM)$. By \cite[Theorem 2.12.1]{Sh1}, 
$$\lambda(\xi): S_{\xi}^m(T^*_x\tM)\to S_{\xi}^m(T^*_x\tM)$$
is an isomorphism for $\xi\neq 0$. Thus there exists $p(\xi)$ such that $\lambda(\xi)p(\xi)=Id$ on $S_{\xi}^m(T^*_x\tM)$. Namely, we can find some pseudodifferential operator $P$ of order $1$, such that on $S_{\xi}^m(T^*_x\tM)$, 
$$NP=E-B$$
for some operator $B$ of order $-1$. Now multiplying both sides with the `solenoidal projection' $E-dS\delta$, which is of order $0$, one has
\begin{equation}\label{order -1}
NP(E-dS\delta)=E-dS\delta-R
\end{equation}
defined on $S^m(T^*\tM)$.

Then we multiply by $\delta$ both sides of \eqref{order -1} to get $\delta R=R'$ with $R'$ some smoothing operator. Let $C=\sum_{k=0}^{\infty}R^k$, it is a pseudodifferential operator of order $0$ and a parametrix for $E-R$. Write \eqref{order -1} as 
$$NP(E-dS\delta)+dS\delta=E-R,$$
and multiply by $C$ both sides to get 
$$NP(E-dS\delta)C+dS\delta+dS\delta\sum_{k=1}^{\infty}R^k=(E-R)C=E+R'',$$
with $R''$ a smoothing operator.
Since $\delta R$ is smoothing, $dS\delta\sum_{k=1}^{\infty}R^k$ is smoothing too. We arrive at the following equation
$$NP(E-dS\delta)C+dS\delta+K=E,$$
where $K$ is a smoothing operator. Denote $P(E-dS\delta)C$ by $Q$ (note that one can make $Q$ properly supported), we get \eqref{parametrix}, which finishes the proof.
\end{proof}

Let $U$ be a small open neighborhood of $M$ in $\tM$. Denote the restriction operator from $\tM$ to $M$ by $r_M$, then the following holds

\begin{Lemma}\label{normal}
Let $M$ be a compact simple Riemannian manifold. Assume $I_m$ is s-injective on $C^{\infty}(S^m(T^*M))$, then the operator
$$r_M N: H^{-1}_c(S^m(T^*\tM))\to L^2(S^m_{sol}(T^*M))$$
is surjective.
\end{Lemma}
Note that elements in $H^{-1}_c(S^m(T^*\tM))$ are defined in the sense of distributions, which are compactly supported in $\tM$.
\begin{proof}
We adopt the approach of \cite{DU} for showing the surjectivity of $N$ on $1$-forms. By Lemma \ref{partial parametrix}, 
$$NQu=u+Ku$$
for all $u\in L^2_c(S^m_{sol}(T^*\tM))$ with $K$ a smoothing operator on $\tM$. Since the simplicity is stable under small $C^2$-perturbations of the metric $g$, by Proposition \ref{extension}, we perturb the metric of $\tM\backslash\overline{M}$ a little bit (still denoted by $g$) so that under the new metric $\tM$ is still simple and there exists a bounded operator $\mathcal{E}: L^2(S^m_{sol}(T^*M))\to L^2(S^m_{U, sol}(T^*\tM))$ such that on $L^2(S^m_{sol}(T^*M))$
$$r_MNQ\mathcal{E}=E+r_MK\mathcal{E}.$$
Since $K$ is a smoothing operator, $r_MK\mathcal{E}$ is compact on $L^2(S^m_{sol}(T^*M))$, which implies that $E+r_MK\mathcal{E}$ has closed range and finite codimension. Thus we have $r_MNQ\mathcal{E}: L^2(S^m_{sol}(T^*M))\to L^2(S^m_{sol}(T^*M))$ has closed range and finite codimension. By the inclusion relation
$$r_MNQ\mathcal{E}(L^2(S^m_{sol}(T^*M)))\subset r_MN(H^{-1}_c(S^m(T^*\tM)))\subset L^2(S^m_{sol}(T^*M)),$$
the intermediate space $r_MN(H^{-1}_c(S^m(T^*\tM)))$ is also closed in $L^2(S^m_{sol}(T^*M))$. Thus it suffices to show that the adjoint $(r_MN)^*$ is injective, which will imply the surjectivity of $r_MN$.

For $L^2$ symmetric $m$-tensor fields, we have the decomposition
\begin{equation}\label{decomposition}
L^2(S^m(T^*M))=L^2(S^m_{sol}(T^*M))\oplus L^2(S^m_{P}(T^*M)),
\end{equation}
where $L^2(S^m_{P}(T^*M))$ is the potential part. 
Thus the dual operator of $r_MN$ is 
$$(r_MN)^*: L^2(S^m_{sol}(T^*M))\to (H^{-1}_c(S^m(T^*\tM)))^*.$$
For $u\in L^2(S^m_{sol}(T^*M))$ and $v\in H^{-1}_c(S^m(T^*\tM))$, if we denote  by $\mathcal E_0u$ the extension of $u$ to $\tM$ by zero (note that generally $\mathcal E_0u$ is not solenoidal on $\tM$), we have
$$((r_MN)^*u,v)=(u, r_MNv)=(\mathcal E_0u,Nv)=(N\mathcal E_0u,v),$$
i.e. $(r_MN)^*=N\mathcal E_0$. 

Therefore given $u\in L^2(S^m_{sol}(T^*M))$, if $N\mathcal E_0u=0$, then
$$0=(N\mathcal E_0u, \mathcal E_0u)=\| I_m\mathcal E_0u\|^2_{L^2(\p_+S{\tM})}\Longrightarrow I_m\mathcal E_0u=0.$$
Since $\mathcal E_0u=0$ outside $M$ and $\tM$ is simple, this implies 
$$I_mu=0.$$
By \cite[Theorem 1.1]{SSU}, $u$ is smooth and $\delta u=0$. The s-injectivity assumption implies $u=0$. This completes the proof of the lemma.
\end{proof}

Next we prove the lemma in the smooth setting:

\begin{Lemma}\label{smooth normal}
Let $M$ be a compact simple Riemannian manifold. Assume $I_m$ is s-injective on $C^{\infty}(S^m(T^*M))$, then the operator
$$r_M N: C^{\infty}_c(S^m(T^*\tM))\to C^{\infty}(S^m_{sol}(T^*M))$$
is surjective.
\end{Lemma}
\begin{proof}
By Lemma \ref{partial parametrix}, 
$$NQu=u+Ku$$
for all $u\in C^{\infty}_c(S^m_{sol}(T^*\tM))$ with $K$ a smoothing operator on $\tM$. Since the simplicity is stable under small $C^2$-perturbations of the metric $g$, by Proposition \ref{smooth extension}, we perturb the metric of $\tM\backslash\overline{M}$ a little bit (still denoted by $g$) so that under the new metric $\tM$ is still simple and there exists a bounded operator $\mathcal{E}: H^k(S^m_{sol}(T^*M))\to L^2(S^m_{U, sol}(T^*\tM))$ for some integer $k\geq 2$ with $\mathcal{E}(C^{\infty}(S^m_{sol}(T^*M)))\subset C^{\infty}(S^m_{U,sol}(T^*\tM))$ such that on $H^k(S^m_{sol}(T^*M))$
$$r_MNQ\mathcal{E}=E+r_MK\mathcal{E}.$$

Now the argument of \cite[Lemma 2.2]{DU} can be applied to tensors of any order to finish the proof.
\end{proof}

\begin{Remark}{\rm One can actually prove Lemma \ref{normal} and \ref{smooth normal} just by applying Lemma \ref{local solenoidal extension}. Given a smooth solenoidal tensor $u$ on $M$, by Lemma \ref{local solenoidal extension} we first extend it to a smooth solenoidal tensor $\tilde u$ on an arbitrarily small open neighborhood $U$, then we extend $\tilde u$ smoothly to $\tM$ with compact support, denoted by $\mathcal E u$. Note that generally $\mathcal E u$ is not solenoidal. Since the Schwartz kernel of the parametrix $S$ of $\delta d$ is smooth away from the diagonal $\Delta_{\tM\times \tM}$, we can choose $S$ to make the support of its Schwartz kernel sufficiently close to $\Delta_{\tM\times \tM}$ so that $dS\delta\mathcal Eu=0$ in an open neighborhood of $M$. This implies that $r_MdS\delta\mathcal Eu=0$, i.e. $r_MNQ\mathcal Eu=u+r_MK\mathcal E u$. It also works for $L^2$ solenoidal tensors.

On the other hand, the original proof of \cite[Lemma 2.2]{DU} uses the existence of compactly supported solenoidal extensions of solenoidal $1$-forms one more time at the very end to show that the adjoint $(r_MN)^*$ is injective. However, one can also avoid this. Notice that given a $1$-form $f$ in the kernel of $(r_MN)^*$, by \cite[equation (2.33)]{DU}, $f=dp$ for some distribution $p$ on $\tM$ with $\mbox{sing\,supp}\,p\subset \p M$ and $p|_{\p\tM}=0$. Moreover, since $\mbox{supp}\,f\subset M$, we have $dp=0$ outside $M$. As $p$ is smooth outside $M$ and $p=0$ on $\p\tM$, strict convexity of $\p M$ implies $p\equiv 0$ in $\tM\backslash M$. Now given a smooth solenoidal $1$-form $u$ on $M$, by Lemma \ref{local solenoidal extension} let $\mathcal Eu$ be the smooth compactly supported extension of $u$ to $\tM$ which is solenoidal in a small open neighborhood ($\neq\tM)$ of $M$. Since the supports of $\delta\mathcal E u$ and $p$ are disjoint, we have
$$(f,\mathcal E u)=(dp,\mathcal E u)=(p,\delta \mathcal E u)=0,$$
which implies that $f=0$, i.e. $(r_MN)^*$ has trivial kernel. The argument works for tensors of arbitrary rank.

At this point, we see that one can prove the surjectivity of $r_MN$ just using Lemma \ref{local solenoidal extension}, without the need of knowing the generic absence of non-trivial Killing tensors \cite{KM}. However, a perturbation of the metric seems still necessary so far for the proof of the existence of {\it compactly supported} solenoidal extensions, and Propositions \ref{extension} and \ref{smooth extension} may find their applications in other areas.} 
\label{rem:colin}
\end{Remark}


\section{Analysis of the adjoint $I_m^*$}
Before proving the main result, we need to extend the definition of the geodesic ray transform $I_m$ so that it acts on negative Sobolev spaces. To this end, we will study the regularity property of the adjoint of the geodesic ray transform, $I_m^*$.

As discussed in the introduction, given $M$ a compact non-trapping manifold with strictly convex boundary, the operator $I_m^*: C^{\infty}_{\alpha}(\p_+SM)\to C^{\infty}(S^m(T^*M))$ is the product of two operators, i.e. $I_m^*=L_m\circ I^*$. 
We instead study the regularity properties of $I^*$ and $L_m$. We start with the latter.

\begin{Lemma}\label{L}
Given a compact Riemannian manifold $M$ (with or without boundary), the operator
$$L_m: H^k(SM)\to H^k(S^m(T^*M))$$
is bounded for every integer $k\geq 0$.
\end{Lemma}
\begin{proof}
Our purpose is to show that there exists a constant $C>0$, such that for any $w\in H^k(SM)$, the following holds
\begin{equation}\label{bounded L}
\|L_mf\|_{H^k}\leq C\|f\|_{H^k}.
\end{equation}
Since $M$ is compact, by a partition of unit, it suffices to show the above inequality in local charts. Let $U$ be a domain in $SM$ with local coordinate system $(z^1,\cdots,z^{2n-1})$. We assume supp$\,f\subset U$. Let $V$ be a domain in $M$ with local coordinate system $(x^1,\cdots,x^n)$, and $\psi$ be a smooth function with support in $V$. We will show
$$\|\psi L_mf\|_{H^k(S^m(T^*V))}\leq C\|f\|_{H^k(U)}.$$
By the definition of the $H^k$-norm of tensors, we only need to show the above inequality is true for each component of the tensor.

We start with $f\in C^{\infty}(SM)$ with support in $U$, then $L_mf$ is also smooth. Let $J=(j_1\cdots j_m)$ and $\xi^J:=\xi^{j_1}\cdots\xi^{j_m}$, then
\begin{equation}
\begin{split}\label{L k-th derivative}
& D^{\alpha}_x[\psi(x)L_mf(x)^J]=D_x^{\alpha}[\psi(x)\int_{S_xM}f(x,\xi)\xi^J\,d\Omega_x(\xi)]\\
=& D_x^{\alpha}[\psi(x)\int_{S^{n-1}}f(x,\xi(x,\eta))\xi^J(x,\eta) P(x,\eta)\,d\Omega(\eta)]\\
=& \sum_{\alpha_1+\alpha_2+\alpha_3=\alpha}D_x^{\alpha_1}\psi(x)\int_{S^{n-1}}D_x^{\alpha_2}f(x,\xi(x,\eta))\cdot D_x^{\alpha_3}[\xi^J(x,\eta) P(x,\eta)]\,d\Omega(\eta)\\
=& \sum_{\alpha_1+\alpha_2+\alpha_3=\alpha}D_x^{\alpha_1}\psi(x)\int_{S_xM}D_x^{\alpha_2}f(x,\xi)\cdot D_x^{\alpha_3}[\xi^J P(x,\eta(x,\xi))]\cdot P'(x,\xi)\,d\Omega_x(\xi).
\end{split}
\end{equation}
Here $P$ and $P'$ are corresponding Jacobians.

For $|\alpha|\leq k$, according to \eqref{L k-th derivative}
\begin{equation*}
\begin{split}
\|D^{\alpha}_x[\psi(x)L_mf(x)_J]\|^2_{L^2(V)} & \leq\sum_{\beta\leq\alpha}C_{\beta,\alpha}\int_V\int_{S_xM}|D_x^{\beta}f(x,\xi)|^2\,d\Omega_x(\xi)dx\\
& \leq \sum_{|\gamma|\leq|\alpha|}C_{\gamma,\alpha}\int_U|D_z^{\gamma}f(z)|^2\,dz\\
& \leq C\|f\|^2_{H^k(U)}.
\end{split}
\end{equation*}
Thus the estimate \eqref{bounded L} is proved when $w\in C^{\infty}(SM)$.

For $f\in H^k(SM)$, since $C^{\infty}(SM)$ is dense in $H^k(SM)$, by an approximation argument, it is easy to show that $L_mf\in H^k(S^m(T^*M))$ and the estimate \eqref{bounded L} holds too. This proves the lemma. 
\end{proof}


Now we turn to the analysis of the operator $I^*$, which basically is an invariant extension, along the geodesic flow, of functions on $\p_+SM$ to functions on $SM$. It is well-known that given $\varphi\in C^{\infty}(\p_+SM)$, $\varphi^{\sharp}=I^*(\varphi)$ is not necessarily in $C^{\infty}(SM)$. The following subspace of $C^{\infty}(\p_+SM)$ has already been considered in the introduction,
$$C^{\infty}_{\alpha}(\p_+SM):=\{\varphi\in C^{\infty}(\p_+SM): \varphi^{\sharp}\in C^{\infty}(SM)\}.$$
In particular, by \cite[Lemma 1.1]{PU}, if $M$ is compact non-trapping with strictly convex boundary,
$$C^{\infty}_{\alpha}(\p_+SM)=\{\varphi\in C^{\infty}(\p_+SM): A\varphi\in C^{\infty}(\p SM)\}$$
where
\begin{equation*}
A\varphi(x,\xi)=\left\{ \begin{array}{l l}
\varphi(x,\xi), \quad & (x,\xi)\in \p_+SM,\\
 \varphi(\gamma_{x,\xi}(-\tau(x,-\xi)),\dot{\gamma}_{x,\xi}(-\tau(x,-\xi))), & (x,\xi)\in \p_-SM .
       \end{array} \right.
\end{equation*}
Since $A\varphi$ is smooth in both $(\p_+SM)^{int}$ and $(\p_-SM)^{int}$, the singularities can only come from $S(\p M)$. We introduce the space $H^{k}_{\alpha}(\p_+SM),\, k\geq 0$ to be the completion of $C^{\infty}_{\alpha}(\p_+SM)$ under the $H^k$-norm. Obviously $H^0_{\alpha}(\p_+SM)=L^2(\p_+SM)$. 
It is easy to show that $C^{\infty}_c((\p_+SM)^{int})\subset C^{\infty}_{\alpha}(\p_+SM)$ (This is from the fact that $\p_+ SM$ is compact and the boundary $\p M$ is strictly convex), which implies that $H^{k}_0(\p_+SM)\subset H^k_{\alpha}(\p_+SM)$. 

\begin{Lemma}\label{phi}
Given a compact non-trapping manifold $M$ with strictly convex boundary, the operator
$$I^*: H^k_{\alpha}(\p_+SM)\to H^k(SM)$$
is bounded for any integer $k\geq 0$.
\end{Lemma}
\begin{proof}
The idea is similar to the proof of Lemma \ref{L}. First we consider the case $\varphi\in C^{\infty}_{\alpha}(\p_+SM)$, thus $\varphi^{\sharp}\in C^{\infty}(SM)$. Let $U$ be a domain in $\p_+SM$ with local coordinate systems $(y^1,\cdots,y^{2n-2})$. We assume supp$\,\varphi\subset U$. Let $V$ be a domain in $SM$ with local coordinate systems $(z^1,\cdots,z^{2n-1})$, and $\psi$ be a smooth function with support in $V$. Since $M$ is compact, it suffices to show
$$\|\psi \varphi^{\sharp}\|_{H^k(V)}\leq C\|\varphi\|_{H^k(U)}.$$

Since
$$D^{\alpha}_z[\psi(z)\varphi^{\sharp}(z)]=\sum_{\beta+\gamma=\alpha}D_z^{\gamma}\psi(z)\cdot D_z^{\beta}\varphi^{\sharp}(z),$$
we obtain that for $|\alpha|\leq k$
\begin{equation*}
\|D^{\alpha}_z[\psi(z)\varphi^{\sharp}(z)]\|^2_{L^2(V)}\leq\sum_{\beta\leq\alpha}C_{\beta,\alpha}\int_V |D_z^{\beta}\varphi^{\sharp}(z)|^2\,dz.
\end{equation*}
Now let $D=\{(y,t): y\in\p_+SM,\,0\leq t\leq\tau(y)\}$ be a closed domain in $\p_+SM\times\mathbb{R}$. Define the map $\Psi: D\to SM$ by $z=\Psi(y,t)=(\gamma_{y}(t),\dot{\gamma}_{y}(t))$, by \cite[Lemma 4.2.2]{Sh1}
\begin{equation*}
\begin{split}
\int_V |D_z^{\beta}\varphi^{\sharp}(z)|^2\,dz & \leq\sum_{|\sigma|+s=|\beta|}C_{\beta, \sigma,s}\int_U\int_0^{\tau(y)}|D_y^{\sigma}D_t^{s}\varphi^{\sharp}(z(y,t))|^2|\<\xi(y), \nu(x(y))\>|\,dtdy\\
& =\sum_{|\sigma|=|\beta|}C_{\beta,\sigma}\int_U\int_0^{\tau(y)}|D_y^{\sigma}\varphi^{\sharp}(y,t)|^2\,dtd\mu(y)\quad (D^s_tD_y^{\sigma}\varphi^{\sharp}=D_y^{\sigma}D^s_t\varphi^{\sharp})\\
& =\sum_{|\sigma|=|\beta|}C_{\beta,\sigma}\int_U\tau(y)|D_y^{\sigma}\varphi(y)|^2\,d\mu(y)\\
& \leq \sum_{|\sigma|=|\beta|}C'_{\beta,\sigma}\int_U|D_y^{\sigma}\varphi(y)|^2\,d\mu(y)\\
& \leq C\|\varphi\|^2_{H^k(U)}.
\end{split}
\end{equation*}
Therefore, $\|\varphi^{\sharp}\|_{H^k(SM)}\leq C\|\varphi\|_{H^k(\p_+SM)}$ for $\varphi\in C^{\infty}_{\alpha}(\p_+SM)$.

If $\varphi\in H^k_{\alpha}(\p_+SM)$, since $C^{\infty}_{\alpha}(\p_+SM)$ is dense in $H^k_{\alpha}(\p_+SM)$, by an approximation argument, it is easy to show that $\varphi^{\sharp}\in H^k(SM)$ and the operator $I^*$ is bounded, which proves the lemma.
\end{proof}

Combining the two lemmas above, we obtain the desired regularity property of $I_m^*$.

\begin{Proposition}\label{I*}
Given a compact non-trapping Riemannian manifold $M$ with strictly convex boundary, the adjoint operator of the geodesic ray transform on symmetric $m$-tensors
$$I_m^*=L_m\circ I^*: H^k_{
\alpha}(\p_+SM)\to H^k(S^m(T^*M))$$
is bounded for any integer $k\geq 0$.
\end{Proposition} 

Now we can extend the definition of the geodesic ray transform so that it acts on $(H^{k}(S^m(T^*M)))^*$ (the dual space is with respect to the $L^2$ inner product) for integers $k\geq 1$. Let $u\in (H^{k}(S^m(T^*M)))^*$ and $\varphi\in H^k_{\alpha}(\p_+SM)$, we define $I_mu$ in the sense of distributions
\begin{equation}\label{I on distributions}
(I_mu, \varphi):=(u, I_m^*\varphi).
\end{equation}
By Proposition \ref{I*}, the right hand side of \eqref{I on distributions} is well-defined. We derive the following corollary:

\begin{Corollary}
Given $M$, a compact non-trapping manifold with strictly convex boundary, the operator
\begin{equation*}
I_m: (H^{k}(S^m(T^*M)))^*\to (H^{k}_{\alpha}(\p_+SM))^*.
\end{equation*}
defined by \eqref{I on distributions} is bounded.
\end{Corollary}
Here the dual space $(H^{k}_{\alpha}(\p_+SM))^*$ is also with respect to the $L^2$ inner product. Notice that $H^{k}_0(\p_+SM)\subset H^k_{\alpha}(\p_+SM)$, thus $(H^{k}_{\alpha}(\p_+SM))^*\subset H^{-k}(\p_+SM)$. On the other hand, since $C^{\infty}(S^m(T^*M))$ is dense in $H^k(S^m(T^*M))$ under the $H^k$-norm, it is clear that $H^{-k}_c(S^m(T^*M^{int}))\subset (H^{k}(S^m(T^*M)))^*$; we will use the weaker map in the next section:
\begin{equation}\label{I on -k}
I_m: H^{-k}_c(S^m(T^*M^{int}))\to H^{-k}(\p_+SM).
\end{equation}


\section{Proof of Theorem \ref{thm:main}}

Now we are in a position to prove our main theorem. We start by showing that (1), (2) and (3) are equivalent.

\begin{proof}
(1)$\Rightarrow$(2): Since $M$ is simple, given $u\in L^2(S^m_{sol}(T^*M))$, by Lemma \ref{normal}, there exists $v\in H^{-1}_c(S^m(T^*\tM))$ such that $r_MI_m^*\,I_mv=u$. Then \eqref{I on -k} implies the existence of some $\tilde{\varphi}=I_mv\in H^{-1}(\p_+S\tM)$, such that $u=r_MI_m^*\tilde{\varphi}$. For $w\in H^1_0(S^m(T^*M))$, we define the distribution $\varphi$ acting on $I_m(H^1_0(S^m(T^*M)))$ by 
$$(\varphi, I_mw):=(\tilde{\varphi}, I_m\tilde{w})=(I_m^*\tilde{\varphi},\tilde{w}),$$
where $\tilde{w}\in H^1_0(S^m(T^*\tM))$ is the extension of $w$ which is zero outside $M$. We claim that there exists $C>0$ such that
$$|(\varphi,I_mw)|\leq C\|I_mw\|_{H^1}$$
for all $w\in H^1_0(S^m(T^*M))$. Assuming the claim, note that $I_mw\in H^1_0(\p_+SM)$ and by the Hahn-Banach Theorem, $\varphi$ can be extended to a bounded linear functional on $H^1_0(\p_+SM)$, still denoted by $\varphi$, i.e. $\varphi\in H^{-1}(\p_+SM)$. By definition of $\varphi$,
$$|(\varphi,I_mw)|=|(\tilde{\varphi},I_m\tilde{w})|\leq C\|I_m\tilde{w}\|_{H^1},$$
therefore to prove the claim, it suffices to show that 
\begin{equation}\label{equivalent norm 1}
\|I_m\tilde{w}\|_{H^1(\p_+S\tM)}\leq C\|I_mw\|_{H^1(\p_+SM)}
\end{equation}
for some $C>0$. 

Assume at this point that inequality \eqref{equivalent norm 1} holds and let us continue with the proof. Now $\varphi\in H^{-1}(\p_+SM)$ is well-defined. Let $w\in H^1_0(S^m(T^*M))$, and let $\tilde{w}$ be  the extension of $w$ into $\tM$ which is zero outside $M$, so $\tilde{w}\in H^1_0(S^m(T^*\tM))$. Then
$$(r_MI_m^*\tilde{\varphi}, w)=(I_m^*\tilde{\varphi}, \tilde{w})=(\tilde{\varphi}, I_m\tilde{w})=(\varphi,I_mw)=(I_m^*\varphi,w).$$
Thus $u=r_MI_m^*\tilde{\varphi}=I_m^*\varphi$. (The choice of $\varphi$ is not unique.)

(2)$\Rightarrow$(3): Given $u\in L^2(S^m_{sol}(T^*M))$, by the assumption, there is $\varphi\in H^{-1}(\p_+SM)$ such that $u=I_m^*\varphi$. Since $I_m^*=L_m\circ I^*$, we define $f=I^* \varphi$, then $f\in H^{-1}(SM)$ and $u=L_mf$. Furthermore, given $h\in H^2_0(SM)$, 
$$(Xf, h)=(f,-Xh)=(I^*\varphi,-Xh)=(\varphi, -I(Xh))=0,$$
i.e. $Xf=0$.

(3)$\Rightarrow$(1): Assume $I_mu=0$ for some $u\in C^{\infty}(S^m_{sol}(T^*M))$, then it is well-known that there exists $h\in C^{\infty}(SM)$ with $h|_{\p SM}=0$ such that
$$Xh=-\ell_m u.$$
Moreover, by \cite[Lemma 2.3]{Sh3} there exists $p\in C^{\infty}(S^{m-1}(T^*M))$ with $p|_{\p M}=0$ such that $u|_{\p M}=dp|_{\p M}$. When $m=0$, this just means $u|_{\p M}=0$. Calculations in local coordinates show that $X(\ell_{m-1}p)=\ell_m dp$, thus we obtain
$$X(h+\ell_{m-1}p)=-\ell_m(u-dp),$$
with $(h+\ell_{m-1}p)|_{\p SM}=0$. 

Under the projection $\pi: SM\to M$, the pullback of the unit normal vector $\nu$ to $\p M$ is the unit normal vector $\mu$ to $\p SM$, and in local coordinates $X=\xi^i\frac{\p}{\p x^i}-\Gamma^i_{jk}\xi^j\xi^k\frac{\p}{\p\xi^i}$, where $\Gamma^i_{jk}$ are the Christoffel symbols. By taking the boundary normal coordinates $(x',x^n)$ near $x\in\p M$ (so $\nu(x)=\mu(x,\xi)=\frac{\p}{\p x^n}$), together with the fact that $(h+\ell_{m-1}p)|_{\p SM}=0$, we obtain that for $(x,\xi)\in \p SM$
$$0=-\ell_m(u-dp)(x,\xi)=X(h+\ell_{m-1}p)(x,\xi)=\xi^n\p_{x^n}(h+\ell_{m-1}p)(x,\xi).$$ 
The first equality comes from the fact $u-dp|_{\p M}=0$. Thus $\p_{\mu}(h+\ell_{m-1}p)(x,\xi)=0,\, \forall \xi\notin S_x\p M$. But since $h$ and $p$ are smooth, and the measure of $S_x\p M$ is zero on $S_xM$, we get $\p_{\mu}(h+\ell_{m-1}p)(x,\xi)=0$ for all $\xi\in S_xM$, so $h+\ell_{m-1}p\in H^2_0(SM)$.  

On the other hand, there exists $f\in H^{-1}(SM)$ with $Xf=0$, such that $u=L_mf$, it follows
$$0=(Xf, h+\ell_{m-1}p)=(f,-X(h+\ell_{m-1}p))=(f,\ell_m(u-dp))=(L_mf,u-dp)=\|u\|^2,$$
the last equality comes from the fact that $u$ is orthogonal to $dp$. Thus $u=0$, which implies the s-injectivity.
\end{proof} 

\begin{Remark}{\rm By carrying out an argument similar to the one of \cite[Lemma 4.1]{SU1}, one can actually show that there exists $p\in C^{\infty}(S^{m-1}(T^*M))$ with $p|_{\p M}=0$ such that $\p^k_{\nu}u|_{\p M}=\p^k_{\nu}dp|_{\p M}$ for all integers $k\geq 0$. When $m=0$, this means the boundary jet of $u$ is zero, i.e. $\p_{\nu}^ku|_{\p M}=0$ for all $k\geq 0$. Note that \cite{SU1} only considers the case that $u$ is a symmetric $2$-tensor fields, but the proof works for tensors of any rank. On the other hand, given $\p^k_{\nu}u|_{\p M}=\p^k_{\nu}dp|_{\p M}$, one should be able to prove that $h+\ell_{m-1}p\in H^{k+2}_0(SM)$ for all $k\geq 0$, i.e. $h+\ell_{m-1}p$ also has zero boundary jet. However, for our purposes $k=0$ is enough.}
\end{Remark}

The thing left to prove is the inequality \eqref{equivalent norm 1}. Actually the $H^k$ norms of $I_mw$ and $I_m\tilde w$ are equivalent for arbitrary $k\geq 0$, provided that $w$ is in $H^k_0(S^m(T^*M))$. A simple calculation shows that $\|I_m\tilde{w}\|_{L^2}^2=(\tilde{w}, I_m^*I_m\tilde{w})=(w, r_MI_m^*I_m\tilde{w})=(w,I_m^*I_mw)=\|I_mw\|_{L^2}^2$. We assume that $\p\tM$ and $\p M$ are sufficiently close.

\begin{Lemma}
Let $M$ be a compact non-trapping manifold with strictly convex boundary. Given $w\in H^k_0(S^m(T^*M))$, $k\geq 1$, let $\tilde{w}\in H^k_0(S^m(T^*\tM))$ be the extension of $w$ to $\tM$ by zero, then there exists $C>1$ such that
\begin{equation}\label{equivalent norm 2} 
\frac{1}{C}\|I_mw\|_{H^k(\p_+SM)}\leq\|I_m\tilde{w}\|_{H^k(\p_+S\tM)}\leq C\|I_mw\|_{H^k(\p_+SM)}.
\end{equation}
\end{Lemma}

\begin{proof}
We only need to show \eqref{equivalent norm 1}, which is half of \eqref{equivalent norm 2}. Since $\p M$ and $\p\tM$ are close, we can assume the closure of $\tM$ is still compact non-trapping with strictly convex boundary. Given a geodesic $\gamma_{x,\xi}$ on $M$ determined by $(x,\xi)\in \p_+SM$, we can uniquely extend it to a geodesic $\gamma_{y,\eta}$ on $\tM$ determined by $(y,\eta)\in \p_+S\tM$. It is not difficult to see that the map
$$T: \p_+SM\to \p_+S\tM,\, \mbox{with}\, \,T(x,\xi)=(y,\eta)$$
is a diffeomorphism from $\p_+SM$ onto its image $T(\p_+SM)$. On the other hand, by the definition of $\tilde{w}$, $I_mw(x,\xi)=I_m\tilde{w}(T(x,\xi))=I_m\tilde{w}(y,\eta)$ and $I_m\tilde{w}(y,\eta)=0$ for $(y,\eta)\in \p_+S\tM\backslash T(\p_+SM)$.  

Since $\p_+SM$ and $\p_+S\tM$ are compact, similar to the proof of Lemma \ref{L} and \ref{phi} we will work in local charts. Let $U$ be a domain in $\p_+S\tM$ with local coordinates $(\tilde z^1,\cdots,\tilde z^{2n-2})$ and $\varphi$ be a smooth function on $\p_+S\tM$ with supp $\varphi\subset U$. In the mean time, there is a domain $V$ in $\p_+SM$ with local coordinates $(z^1,\cdots,z^{2n-2})$ such that $T^{-1}(U\cap T(\p_+SM))\subset V$, and $\psi$ is a smooth function on $\p_+SM$ with $T^{-1}(U\cap T(\p_+SM))\subset$ supp $\psi\subset V$, $\psi\equiv 1$ on $T^{-1}(U\cap T(\p_+SM))$. We first consider the case $w\in C^{\infty}_c(S^m(T^*M)^{int})$ and show that there exists $C>0$ such that
$$\|\varphi\cdot I_m\tilde w\|_{H^k(U)}\leq C\|\psi\cdot I_mw\|_{H^k(V)}.$$

Notice that for $|\alpha|\leq k$ 
$$D_{\tilde z}^{\alpha}[\varphi\cdot I_m\tilde w]=\sum_{\beta+\gamma=\alpha}D_{\tilde z}^{\gamma}\varphi\cdot D_{\tilde z}^{\beta} I_m\tilde w,$$
thus
\begin{equation*}
\begin{split}
\|D_{\tilde z}^{\alpha}[\varphi\cdot I_m\tilde w]\|^2_{L^2(U)} &\leq\sum_{\beta\leq\alpha}C_{\beta,\alpha}\,\int_U |D_{\tilde z}^{\beta} I_m\tilde w|^2\,d\tilde z\\
& =\sum_{\beta\leq\alpha}C_{\beta,\alpha}\,\int_{U\cap T(\p_+SM)} |D_{\tilde z}^{\beta} I_m\tilde w(\tilde z)|^2\,d\tilde z\\
&\leq \sum_{|\sigma|\leq|\alpha|}C_{\sigma,\alpha}\,\int_{T^{-1}(U\cap T(\p_+SM))} |D_{z}^{\sigma} I_m\tilde w(T(z))|^2J\,dz\\
&\leq C'\sum_{|\sigma|\leq |\alpha|}\int_{T^{-1}(U\cap T(\p_+SM))} |D_{z}^{\sigma} (\psi\cdot I_m w)(z)|^2\,dz\\
&\leq C'\sum_{|\sigma|\leq |\alpha|}\int_{V} |D_{z}^{\sigma} (\psi\cdot I_m w)(z)|^2\,dz\leq C\|\psi\cdot I_mw\|^2_{H^k(V)},
\end{split}
\end{equation*}
here $J$ is the Jacobian related to the diffeomorphism $T$. Therefore $\|I_m\tilde{w}\|_{H^k(\p_+S\tM)}\leq C\|I_mw\|_{H^k(\p_+SM)}$ for $w\in C^{\infty}_{c}(S^m(T^*M)^{int})$.

Now for $w\in H^k_0(S^m(T^*M))$, there is a sequence $w_k\in C^{\infty}_c(S^m(T^*M)^{int})\, k=1,2,\cdots$, which converges to $w$ in the $H^k$-norm. Then it is not difficult to see that the sequence $\tilde w_k\in C^{\infty}_c(S^m(T^*\tM))$ converges to $\tilde w\in H^k_0(S^m(T^*\tM))$. By the boundedness of the operator $I_m$, $I_mw_k$ and $I_m\tilde w_k$ converge to $I_mw$ and $I_m\tilde w$ respectively in $H^k$-norm. This implies that above estimates are valid for any $w\in H^k_0(S^m(T^*M))$.  
\end{proof}


The following proposition that holds on compact non-trapping manifolds with strictly convex boundary shows that items
(4) and (5) in Theorem \ref{thm:main} are equivalent and any of them implies item (1). 
\begin{Proposition}
Let $M$ be a compact non-trapping Riemannian manifold with strictly convex boundary and let $u\in C^{\infty}(S^m_{sol}(T^*M))$. The following are equivalent:\\
(i) there exists $\varphi\in C^{\infty}_{\alpha}(\p_+SM)$ such that $u=I_{m}^*\varphi$;\\
(ii) there exists $f\in C^{\infty}(SM)$ satisfying $Xf=0$ and 
$u=L_{m}f.$

Any of these two conditions implies s-injectivity of $I_m$.
\end{Proposition}

\begin{proof}
(i)$\Rightarrow$(ii): By the assumption, there is $\varphi\in C^{\infty}_{\alpha}(\p_+SM)$ such that $u=I^*_m\varphi=L_m\circ I^*\varphi$. Define $f=I^*\varphi=\varphi^{\sharp}\in C^{\infty}(SM)$ (since $\varphi\in C^{\infty}_{\alpha}(\p_+SM)$), then $u=L_mf$. Moreover, it is clear that $Xf=X\varphi^{\sharp}=0$ by definition.

(ii)$\Rightarrow$(i): If there exists $f\in C^{\infty}(SM)$ with $Xf=0$, this implies that $f=I^*(f|_{\p_+SM})$. We define $\varphi=f|_{\p_+SM}\in C^{\infty}(\p_+SM)$. However, since $\varphi^{\sharp}=f\in C^{\infty}(SM)$, $\varphi$ actually sits in the space $C^{\infty}_{\alpha}(\p_+SM)$. By the assumption, $u=L_mf=L_m\circ I^*\varphi=I_m^*\varphi$.

The argument that shows that any of these conditions imply s-injectivity of $I_m$ is even easier than the proof
that (3) implies (1) in Theorem \ref{thm:main} since we do not have to worry about paring $Xf$ with an element in
$H_{0}^{2}(SM)$.  Assuming (ii), integration by parts yields right away:
$$0=(Xf, h)=(f,-Xh)=(f,\ell_m(u))=(L_mf,u)=\|u\|^2.$$

\end{proof}

Finally we show that in Theorem \ref{thm:main}, item (1) implies item (4):

Since $M$ is simple, given $u\in C^{\infty}(S^m_{sol}(T^*M))$, by Lemma \ref{smooth normal}, there exists $v\in C^{\infty}_c(S^m(T^*\tM))$ such that $r_MI_m^*\,I_mv=u$. Then it is a standard argument that if we define $\varphi=I^*(I_mv)|_{\p_+SM}$, then $I^*_m\varphi=u$. Moreover, since $I^*(I_mv)$ is smooth in the interior of $S\tM$, we have $\varphi\in C^{\infty}_{\alpha}(\p_+SM)$.

The proof of Theorem \ref{thm:main} is now complete.


\section{Alternative proof of Corollary \ref{thm:2D}}

Before giving the alternative proof, we will explain how the solenoidal condition of a tensor manifests itself at the level of the transport equation. It seems that this basic relation has not appeared
before in the literature, although we believe it was known to experts.

As we already pointed out in the introduction, by considering the vertical Laplacian $\Delta$ on each fibre $S_{x}M$ of $SM$ we have a natural $L^2$-decomposition $L^{2}(SM)=\oplus_{m\geq 0}H_m(SM)$
into vertical spherical harmonics. We set $\Omega_m:=H_{m}(SM)\cap C^{\infty}(SM)$. Then a function $u$ belongs to $\Omega_m$ if and only if
$\Delta u=m(m+n-2)u$ where $n=\dim M$. 
The maps
$$\ell_m: C^{\infty}(S^m(T^*M))\to \bigoplus_{k=0}^{[m/2]}\Omega_{m-2k},$$ 
and
$$L_{m}:\bigoplus_{k=0}^{[m/2]}\Omega_{m-2k}\to C^{\infty}(S^m(T^*M))$$
are isomorphisms.
These maps give natural identification between functions in $\Omega_m$ and {\it trace-free} symmetric $m$-tensors (for details on this see \cite{GK2,DS,PSU5}). The geodesic vector field $X$ maps $\Omega_m$ to
$\Omega_{m-1}\oplus\Omega_{m+1}$ and hence we can split it as $X=X_{+}+X_{-}$, where $X_{\pm}:\Omega_{m}\to \Omega_{m\pm 1}$
and $X_{+}^*=-X_{-}$.  Note that
\[X\ell_{m-1}=\ell_{m}d.\]

Given $f\in \bigoplus_{k=0}^{[m/2]}\Omega_{m-2k}$, in general 
 $Xf\in \bigoplus_{k=0}^{[(m+1)/2]}\Omega_{m+1-2k}$. The next simple lemma characterizes the solenoidal condition in terms of $Xf$.
 
 \begin{Lemma} $Xf\in \Omega_{m+1}$, if and only if $L_{m}f$ is a solenoidal tensor.
 \label{lemma:aux0}
 \end{Lemma}
 
 \begin{proof} Note that $L_{m}f$ is solenoidal if and only if $(L_{m}f,dh)=0$ for any $h\in C^{\infty}(S^{m-1}(T^*M))$ with
 $h|_{\partial M}=0$. But 
\[ (L_{m}f,dh)=(f,\ell_{m}dh)=(f,X\ell_{m-1}h)=-(Xf,\ell_{m-1}h)\]
and the last term is zero if and only if $(Xf)_{m-2k-1}=0$ for $0\leq k\leq [(m-1)/2]$ since
$\ell_{m-1}h\in \bigoplus_{k=0}^{[(m-1)/2]}\Omega_{m-1-2k}$.
\end{proof}

Another way to look at the condition $Xf\in\Omega_{m+1}$ is that the following equations should hold:
\[X_{-}f_{m-2k}+X_{+}f_{m-2k-2}=0,\;\;\;\;\text{for}\;0\leq k\leq [(m-1)/2].\]

\begin{Lemma} The following are equivalent:
\begin{enumerate}

\item Given a non-negative integer $m$ and $a_m\in \Omega_m$ with $X_{-}a_{m}=0$, there exists
$w\in C^{\infty}(SM)$ such that $Xw=0$ and $w_{m}=a_{m}$.
\item Given a non-negative integer $m$ and $f=\sum_{k=0}^{m}f_{k}$ such that $Xf\in\Omega_{m}\oplus\Omega_{m+1}$, there exists $w\in C^{\infty}(SM)$ such that $Xw=0$ and $\sum_{k=0}^{m}w_{k}=f$.

\end{enumerate}
\label{lemma:aux1}
\end{Lemma}

\begin{proof} The fact that (2) implies (1) is quite obvious from the fact that $a_{m}\in\Omega_{m}$ with $X_{-}a_{m}=0$ implies $Xa_{m}=X_{+}a_{m}\in\Omega_{m+1}$.

To prove that (1) implies (2) we proceed by induction in $m$. The case $m=0$ follows right away since $Xf_{0}\in\Omega_{1}$ and $X_{-}f_{0}=0$.

Suppose the claim holds for $m$ and let $f=\sum_{k=0}^{m+1}f_k$ be given with
$Xf\in\Omega_{m+1}\oplus\Omega_{m+2}$. This is equivalent to saying that
$X(\sum_{k=0}^{m}f_{k})\in \Omega_{m}\oplus\Omega_{m+1}$ and
$X_{-}f_{m+1}+X_{+}f_{m-1}=0$.

By induction hypothesis, there exists $w\in C^{\infty}(SM)$ such that $Xw=0$ and 
$w_{k}=f_{k}$ for all $k\leq m$. The equation $Xw=0$ in degree $m$ is
\[X_{-}w_{m+1}+X_{+}f_{m-1}=0\]
and thus
\[X_{-}(f_{m+1}-w_{m+1})=0.\]
Using item (1) in the lemma, there exists $w'=\sum_{m+1}^{\infty}w'_{k}\in C^{\infty}(SM)$
such that $Xw'=0$ and $w'_{m+1}=f_{m+1}-w_{m+1}$.
Then $X(w+w')=0$ and $\sum_{k=0}^{m+1}(w+w')_{k}=f$ as desired.
\end{proof}

Finally we show:

\begin{Proposition} The following are equivalent:

\begin{enumerate}
\item  Given a non-negative integer $m$ and $u\in C^{\infty}(S^{m}_{sol}(T^*M))$, 
there exists $f\in C^{\infty}(SM)$ with $Xf=0$ such that $L_{m}f=u$.

\item Given a non-negative integer $m$ and $a_m\in \Omega_m$ with $X_{-}a_{m}=0$, there exists
$w\in C^{\infty}(SM)$ such that $Xw=0$ and $w_{m}=a_{m}$.

\end{enumerate}
\label{prop:anydim}
\end{Proposition}

\begin{proof} Assume (1) holds. Given $a_{m}\in\Omega_{m}$ with $X_{-}a_{m}=0$ we see using
Lemma \ref{lemma:aux0} that $L_{m}a_{m}$ is a solenoidal tensor. Hence there is $f$
such that $Xf=0$ and $f_m=L_{m}^{-1}L_{m}f=a_{m}$ (note that $L_mf_k=0$ for $k>m$). Thus (2) holds.

Conversely if (2) holds, then item (2) in Lemma \ref{lemma:aux1} holds. Thus there exists
$f\in C^{\infty}(SM)$ such that $Xf=0$ and $\sum_{k=0}^{[m/2]}f_{m-2k}=L_{m}^{-1}u$
and (1) holds.
\end{proof}

\subsection{Proof of Corollary \ref{thm:2D}} On account of Proposition \ref{prop:anydim} it suffices
to show that given $a_m\in \Omega_m$ with $X_{-}a_{m}=0$, there exists
$w\in C^{\infty}(SM)$ such that $Xw=0$ and $w_{m}=a_{m}$.  What makes this possible in dimension
two is \cite[Lemma 5.6]{PSU3} whose content we now explain.

If $(M,g)$ is an oriented Riemannian surface, there is a global orthonormal frame $\{ X, X_{\perp}, V \}$ of $SM$ equipped with the Sasaki metric, where $X$ is the geodesic vector field, $V$ is the vertical vector field and $X_{\perp} = [X, V]$. We define the Guillemin-Kazhdan operators \cite{GK}
$$
\eta_{\pm} = \frac{1}{2}(X \pm i X_{\perp}).
$$
If $x = (x_1, x_2)$ are oriented isothermal coordinates near some point of $M$, we obtain local coordinates $(x, \theta)$ on $SM$ where $\theta$ is the angle between $\xi$ and $\partial/\partial x_1$. In these coordinates $V = \partial/\partial \theta$ and $\eta_+$ and $\eta_-$ are $\partial$ and $\overline{\partial}$ type operators, see \cite[Appendix B]{PSU5}.

For any $m \in \mathbb{Z}$ we define 
$$
\Lambda_m = \{ u \in C^{\infty}(SM) \,;\, Vu = imu \}.
$$
In the $(x, \theta)$ coordinates elements of $\Lambda_m$ look locally like $h(x) e^{im\theta}$. Spherical harmonics may be further decomposed as 
\begin{gather*}
\Omega_0 = \Lambda_0, \\
\Omega_m = \Lambda_m \oplus \Lambda_{-m} \text{ for } m \geq 1.
\end{gather*}
Any $u \in C^{\infty}(SM)$ has a decomposition $u = \sum_{m=-\infty}^{\infty} u_m$ where $u_m \in \Lambda_m$. The geodesic vector field decomposes as 
$$
X = \eta_+ + \eta_-
$$
where $\eta_{\pm}: \Lambda_m \to \Lambda_{m \pm 1}$.
If $m \geq 1$, the action of $X_{\pm}$ on $\Omega_m$ is given by 
$$
X_{\pm}(e_m+e_{-m}) = \eta_{\pm} e_m + \eta_{\mp} e_{-m}, \quad e_j \in \Lambda_j,
$$
and for $m=0$ we have $X_+|_{\Omega_0} = \eta_+ + \eta_-$, $X_-|_{\Omega_0} =0$. 

With these preliminaries out of the way, \cite[Lemma 5.6]{PSU3} says that given $f\in\Lambda_m$ there is a smooth $w\in C^{\infty}(SM)$
with $Xw=0$ and $w_m=f$. For $m=0$ this gives the desired result right away.

Given $a_{m}\in\Omega_{m}$ with $X_{-}a_{m}=0$ and $m\geq 1$, we write $a_{m}=e_{m}+e_{-m}$ with $e_{j}\in \Lambda_{j}$. Then $\eta_{-}e_{m}+\eta_{+}e_{-m}=0$. Consider now smooth $p,q$ with $Xp=Xq=0$ and $p_{m}=e_{m}$ and $q_{-m}=e_{-m}$. Then
\[w=\sum_{-\infty}^{-m}q_{k}+\sum_{m}^{\infty}p_{k}\]
satisfies $Xw=0$ and $w_{m}=a_{m}$.
\qed


\end{document}